\documentclass[11pt]{article}
\usepackage[centertags]{amsmath}
\usepackage{amsmath}
\usepackage{hyperref}
\usepackage{amssymb}
\usepackage{amsthm}
\usepackage{color}
\usepackage{latexsym}
\usepackage[T1]{fontenc}
\usepackage[latin1]{inputenc}
\usepackage{fancyhdr}
\usepackage{indentfirst}
\usepackage{times}
\usepackage[title]{appendix}
\DeclareMathOperator*{\esssup}{ess\,sup}

\numberwithin{equation}{section}

\newcommand{\be}{\begin{equation}}
\newcommand{\ee}{\end{equation}}
\newcommand{\beqa}{\begin{eqnarray}}
\newcommand{\eeqa}{\end{eqnarray}}
\newcommand{\beq}{\begin{eqnarray*}}
	\newcommand{\eeq}{\end{eqnarray*}}
\newcommand{\ba}{\begin{aligned}}
	\newcommand{\ea}{\end{aligned}}
\def \ed {\end{document}}

\newtheorem{definition}{Definition}[section]
\newtheorem{theorem}{Theorem}[section]
\newtheorem{proposition}{Proposition}[section]
\newtheorem{lemma}{Lemma}[section]

\newtheorem{corollary}{Corollary}[section]

\def \esssup {\mbox{ess sup}}

\def \R{\mathbb{R}}
\def \P{\mathbb{P}}

\def \E{\mathbb{E}}
\def \Sm{\mathbb{S}}

\def \F{\mathcal{F}}

\def \K{\mathcal{K}}
\def \M{\mathcal{M}}
\def \S{\mathcal{S}}
\def \I{\mathcal{I}}

\def \L{\mathcal{L}}

\def \D{\mathcal{D}}
\def \Y{\mathcal{Y}}

\def\cL{{\cal L}}
\def \bf{\textbf}
\def \it{\textit}

\def \ms {\medskip}

\def \t {\tau}

\def \nn{\nonumber}
\begin{document}
% % % % % % % % % % % % % % % % % % % % % % % % % % % % % % % % % % % % % % % % % % % % % % % % % % % % % % % % % % % % % % % % % % % % % % % % % % % % % % % % %

\title{Stochastic optimal switching and systems of variational inequalities with interconnected obstacles
%: the probabilistic approach
}
% % % % % % % % % % % % % % % % % % % % % % % % % % % % % % % % % % % % % % % % % % % % % % % % % % % % % % % % % % % % % % % % % % % % % % % % % % % % % % % % %

\author{ Brahim EL ASRI \thanks{Universit\'e Ibn Zohr, Labo. LISAD, Equipe. Aide \`a la decision,
		ENSA, B.P.  1136, Agadir, Maroc. E-mail: b.elasri@uiz.ac.ma } ,\,\,\,  Imade FAKHOURI \thanks{Universit\'e Ibn Zohr, Labo. LISAD, Equipe. Aide \`a la decision,
		ENSA, B.P.  1136, Agadir, Maroc. E-mail: i.fakhouri@uiz.ac.ma } \,\,\, and \, Nacer OURKIYA \thanks{Universit\'e Ibn Zohr, Labo. LISAD, Equipe. Aide \`a la decision,
		ENSA, B.P.  1136, Agadir, Maroc. E-mail: nacer.ourkiya@edu.uiz.ac.ma.}}
\date{}
\maketitle

\begin{abstract}
	This paper studies a  system of $m$ variational inequalities with interconnected obstacles in infinite horizon associated to optimal multi-modes switching problems. Our main result is the existence and uniqueness of a continuous solution in viscosity sense, for that system. The proof of the main result strongly relies on the connection between the systems of variational inequalities and reflected backward stochastic differential equations (RBSDEs) with oblique reflection, which will be characterized through a Feynman-Kac's formula. The main feature of our system of infinite horizon RBSDEs is that its components are interconnected through both the generators and the obstacles.
\end{abstract}

\\


\medskip
% % % % % % % % % % % % % % % % % % % % % % % % % % % % % % % % % % % % % % % % % % % % % % % % % % % % % % % % % % % % % % % % % % % % % % % % % % % % % % % % %
\section{Introduction}
In this paper we study the system: for $i\in\I:=\{1,\dots,m\}$ and $x\in\R^k$,
\begin{equation}\label{ourPDE}
\begin{aligned}
\min\biggr\{v^i(x)- &\max\limits_{j\in{\I}^{-i}}\{-g_{ij}(x)+v^j(x)\},rv^i(x)-\cL v^i(x)\\ &-f_i\big(x,v^1(x),...,v^m(x),\sigma^\top(x).D_xv^i(x)\big)\biggl\}=0,
\end{aligned}
\end{equation}
where $r$ is a positive discount factor,  $\I^{-i}:=\I-\{i\}$. $g_{ij}$ and $f_i$ are deterministic continuous functions. $\cL$ is an infinitesimal generator of the form
\be
\label{derivgen} \cL=\frac{1}{2}\sum_{i,j=1}^{m}(\sigma
\sigma^{\top})_{ij}(x)\frac{\partial^{2}}{\partial x_{i} \partial
	x_{j}}+\sum_{i=1}^{m}b_{i}(x) \frac{\partial }{\partial x_{i}},
\ee
where $b_i$ and $\sigma_{ij}$ are continuous functions. Operators $\cL$
occur, for instance, in the context of financial markets where the state variables are
described by an k-dimensional stochastic process $X:=(X^x_t)_{t\geq0}$ solving the system of
stochastic differential equations:
%$(X^x_t)_{t\geq0}$ which is the solution of the following standard stochastic differential equation (SDE):
\be\label{SDE}
dX_t^x=b(X^x_t)dt+\sigma(X^x_t)dB_t, \quad X^x_0=x, x\in\R^k, t\geq 0.
\ee
Here $B:=(B_t)_{t\geq0}$ is a $d$-dimensional Brownian motion. The process $X:=(X^x_t)_{t\geq0}$ can, for instance, be the electricity price, functionals of the electricity price, or other factors which determine the price. $\cL$ is, in the context of \eqref{SDE}, the infinitesimal
generator associated with $X:=(X^x_t)_{t\geq0}$. \\
In the Markovian setting when the randomness stems from the Ito diffusion $X:=(X^x_t)_{t\geq0}$ in \eqref{SDE}, if for $i\in\I$, $f_i$ does not depend on $(y_i)_{i\in\I}$ and $(z_i)_{i\in\I}$, the problem in \eqref{ourPDE} is a system of $m$-variational inequalities with interconnected obstacles occuring in the context of infinite horizon multi-modes optimal switching problems (OSP). Our main results concern the existence and uniqueness of a solution $(v^1,\dots,v^m): x\in \R^k\mapsto (v^1(x),\dots,v^m(x))\in \R^{m} $ in viscosity sense to the system \eqref{ourPDE}.\\
The system \eqref{ourPDE} is connected with the following system of RBSDEs with oblique reflection:
for $i\in\I$ and $t\geq0$, $\forall r\in\R^+$:
\be\label{RBSDE-Markov-framework}
\begin{cases}
	\begin{aligned}
		e^{-r t}Y^{i,x}_t = \int_{t}^{+\infty} &e^{-r s}f_i(X^{x}_s,
		Y^{1,x}_s,\dots,Y^{m,x}_s,Z^{i,x}_s)ds+ K^{i,x}_\infty-K^{i,x}_t\\ &-  \int_{t}^{+\infty}e^{-r s}Z^{i,x}_sdB_s;
	\end{aligned}\\
	\lim\limits_{t \to+\infty} e^{-r t}Y^{i,x}_t=0,\\
	\forall t\geq0, \quad e^{-r t}Y^{i,x}_t \geq e^{-r t} \underset{j\in\I^{-i}}{\max}(Y^{j,x}_t-g_{ij}(X^{x}_t)),\\
	\int_{0}^{+\infty} e^{-r s}\{Y^{i,x}_s - \underset{j\in\I^{-i}}{\max}(Y^{j,x}_s-g_{ij}(X^{x}_s))\}dK^{i,x}_s = 0.
\end{cases}
\ee
The main feature of this system, is that its components $(Y^{1,x},\ldots,Y^{m,x})$ are interconnected through both the barriers and the generators. Actually, the dependence of $f_i$ on the whole vector $(Y^{1,x},\ldots,Y^{m,x})$ can be interpreted as a nonzero-sum game problem, where the players' utilities affect each other.\\
Note that, in a non Markovian framework this system has already been studied in \cite{EO}.

Finite horizon OSP has been widely studied during the last decades, since it can be related to many practical applications, for example the problem of valuation investment opportunities, in particular in energy markets. The OSP, called also reversible investment problem, is a particular case of real options where one is concerned with choices faced by businesses. Examples of such decisions include the option to defer, abandon or
alter projects, to expand or contract business lines, or to switch inputs or outputs for the company's production.
In the literature concerning the OSP three different approaches have turned out to be useful, two based on stochastic techniques (Snell envelope and reflected backward stochastic differential equations) and one on deterministic methods based on systems of variational inequalities (see e.g. \cite{CEK,EF,EH,HJ,HM,HT,HZ} and the references therein). \\
In the finite horizon framework, the closest paper to ours is the one by Hamad\`ene and Morlais \cite{HM} where the authors deal with the existence and uniqueness of solutions in viscosity sense to the following system of variational inequalities: $\forall i\in\I$, $\forall t\leq T$
%have shown the existence and uniqueness of a continuous solution in viscosity sense for the following system: $\forall i\in\I$, $\forall t\leq T$
\be\label{finiPDE}
\begin{cases}
	\ba
	&\min\biggr\{v^i(t,x)- \max\limits_{j\in{\I}^{-i}}\{-g_{ij}(t,x)+v^j(t,x)\},-\partial_t v^i(t,x)-\cL v^i(t,x)\\ &\qquad\qquad-f_i\big(t,x,(v^1,...,v^m)(t,x),\sigma^\top(t,x).D_xv^i(t,x)\big)\biggl\}=0,\\ & v^i(T,x)=h_i(x).
	\ea
\end{cases}
\ee
This system is the analogue of our system \eqref{ourPDE} in a finite horizon framework.\\
It is worth noting that in the context of infinite horizon OSP considerably less work seems to be done and the authors are only aware of the works \cite{Elasri,EO}.

To briefly outline the setting of OSP let us deal with an example. In conciseness, such a problem consists in finding an optimal management strategy for a production company that can run in $m$, $m\geq2$, different modes. Denote $\I$ %=\{1,\ldots,m\}$
the set of possible production modes.
Let $X:=(X^x_t)_{t\geq0}$ be a vector-valued Markovian stochastic process defined in \eqref{SDE}, which represents random factors that influence the profitability of the production. Let the running payoff in production mode $i$, at time $t$, be $f_i(X^x_t)$ and let $g_{ij}(X^x_t)$ denote the cost of switching from mode $i$ to mode $j$ at time $t$. A management strategy %$\delta$ %=\{(\tau_{n})_{n\geq 0},(\epsilon_{n})_{n\geq 0}\}$
is a combination of a nondecreasing sequence of stopping times $(\tau_{n})_{n\geq 0}$, and a sequence of random variables $(\zeta_{n})_{n\geq 0}$ taking values in the set $\I$. At time $\t_n$, in order to maximize the profit of the company, the manager decides to switch the production from the current mode $\zeta_{n-1}$ to $\zeta_{n}$. A strategy $\big(\left\{\tau_{k}\right\}_{k \geq 0},\left\{\zeta_{k}\right\}_{k \geq 0}\big)$ can be represented by the function $a:\R^+ \rightarrow \I$ defined as
$$
a_{s} \equiv a(s)=\sum_{k\geq 0} \zeta_{k} \mathbb{I}_{\left[\tau_{k}, \tau_{k+1}\right)}(s),
$$
indicating the current state of the facility at time $s$. Using the dynamic programming principle, it is well known that the expected profit yield while in mode $i$, $Y^{i}_t$, which can be also seen as the value of an optimal switching strategy, can be expressed in terms of Snell envelope as follows (see \cite[Proposition 1]{Elasri}):
\begin{equation}\label{optimal-value-mode-i}
e^{-rt}Y^{i}_t = \underset{(\tau_{k},\zeta_{k})_{k\geq 0}\in\D_{t}^{i}}{\esssup}\;\E\biggl[\int_{t}^{+\infty} e^{-rs}f_{a_s}(X_s) \, ds-\sum_{n\geq 1}e^{-r\t_n}g_{a_{\t_{n-1}}a_{\t_n}}(X_{\t_n})
%A_{\infty}^{a}
|\mathcal{F}_t\biggr];
\end{equation}
where $\D_{t}^{i}=\left\{(\tau_{k},\zeta_{k})_{k\geq0}  \mbox{ such that } \tau_{0}=t \mbox{ and } \zeta_{0}=i\right\}$. %, %the process $(a_t)_{t\geq0}$ is indicating the mode of the system at time s
%and $A_{\infty}^{a}$ stands for the total switching cost when the strategy $a$ is implemented.
%$Y^{i}_t$ can also be seen as the value of an optimal switching strategy.\\
When the production of the company is working under a strategy $a$, it generates a total expected profit equal to
\begin{equation}
J(x,a) = \E\biggl[\int_{0}^{+\infty} e^{-rs}f_{a_s}(X^x_s) \, ds-\sum_{n\geq 1}e^{-r\t_n}g_{a_{\t_{n-1}}a_{\t_n}}(X_{\t_n})%A_{\infty}^{a}
\biggr].
\end{equation}
The task in the OSP is to find an optimal management strategy $a^*$ such that
$$J(\delta^*)=Y_0^i=\sup\limits_{a}J(x,a)=\underset{(\tau_{k},\zeta_{k})_{k\geq 0}\in\D_{0}^{i}}{\sup}\;\E\biggl[\int_{0}^{+\infty} e^{-rs}f_{a_s}(X^x_s) \, ds-\sum_{n\geq 1}e^{-r\t_n}g_{a_{\t_{n-1}}a_{\t_n}}(X_{\t_n})
%A_{\infty}^{a}
\biggr].$$
%OSP consists in finding an optimal management strategy for a production company that can run in m, $m\geq2$, different modes . %, in order to maximize the golobal profit of the company.
%A management strategy $\delta$ %=\{(\tau_{n})_{n\geq 0},(\epsilon_{n})_{n\geq 0}\}$
%is a combination of a nondecreasing sequence of stopping times $(\tau_{n})_{n\geq 0}$, and a sequence of random variables $(\epsilon_{n})_{n\geq 0}$ taking values in the set of possible production modes $\Lambda=\{1,\ldots,m\}$. At time $\t_n$, in order to maximize the profit of the company, the manager decides to switch the production from the current mode $\epsilon_{n-1}$ to $\epsilon_{n}$. When the production of the company is working under a strategy $\delta$, it generates a gain equal to $J(\delta)$. The OSP amounts to finding an optimal management strategy $\delta^*$ such that $J(\delta^*)=\Sup_{\delta}J(\delta)$.

It was proven in \cite{Elasri} that the optimal value vector $(Y^1_t,\ldots,Y_t^m)$ defined in \eqref{optimal-value-mode-i} solves the
following system of reflected backward stochastic differential equation with oblique reflection:
%The infinite horizon OSP is related to the following system of variational inequalities:
%The system \eqref{ourPDE} is connected with the following system of RBSDEs with oblique reflection:
for $i\in\I$ and $t\geq0$, $\forall r\in\R^+$:
\be\label{RBSDE-Markov-framework-elasri}
\begin{cases}
	\begin{aligned}
		e^{-r t}Y^{i}_t = \int_{t}^{+\infty} &e^{-r s}f_i(X^{x}_s)ds+ K^{i,x}_\infty-K^{i}_t-  \int_{t}^{+\infty}e^{-r s}Z^{i}_sdB_s;
	\end{aligned}\\
	\lim\limits_{t \to+\infty} e^{-r t}Y^{i}_t=0,\\
	\forall t\geq0, \quad e^{-r t}Y^{i}_t \geq e^{-r t} \underset{j\in\I^{-i}}{\max}(Y^{j}_t-g_{ij}(X^{x}_t)),\\
	\int_{0}^{+\infty} e^{-r s}\{Y^{i}_s - \underset{j\in\I^{-i}}{\max}(Y^{j}_s-g_{ij}(X^{x}_s))\}dK^{i}_s = 0,
\end{cases}
\ee
Moreover, the author showed that, if the switching costs satisfy $\frac{1}{\alpha}\leq g_{ij}(x)\leq \alpha$ where $\alpha$ is a strictly positive constant, then there exist deterministic continuous functions of polynomial growth $v^{1}, \ldots, v^{m}: \mathbb{R}^{k} \rightarrow \mathbb{R}$ such that:
$$
\forall x \in\R^k,\; Y_{0}^{i}=v^{i}(x),\; i=1, \ldots, m,
$$
where the functions $v^{i}, i=1, \ldots, m$, are unique solutions in viscosity sense of the
following system of variational inequalities with interconnected obstacles: for $i\in\I$
\begin{equation}\label{ELasri-IH-PDE}
%\begin{aligned}
\min\biggr\{v^i(x)- \max\limits_{j\in{\I}^{-i}}\{-g_{ij}(x)+v^j(x)\},rv^i(x)-\cL v^i(x)-f_i\big(x\big)\biggl\}=0.
%\end{aligned}
\end{equation}

%Let us close this introduction by shedding some light on the main contributions of this paper.
The main result of our paper is the existence and uniqueness of the solution in viscosity sense of the system \eqref{ourPDE} (see Theorem \ref{Theo_Ex_uni}), which still remains an open problem. The main tool to tackle the system \eqref{ourPDE} is to deal with the related system of RBSDEs \eqref{RBSDE-Markov-framework}. Actually we show that both systems \eqref{ourPDE} and \eqref{RBSDE-Markov-framework} are linked through the following relation (see Theorem \ref{therem31}), the so-called Feynman-Kac formula:
\begin{equation}\label{Intro-Feynman-Kac}
\forall x \in \mathbb{R}^{k},\; Y_{0}^{i,x}=v^{i}(x),\; i=1, \ldots, m.
\end{equation}
In order to prove the main result, we need some results which to the best of our knowledge do not exist in the literature. Indeed, we prove the following auxiliary results:
\begin{itemize}
	\item[$\bullet$] Existence and uniqueness of the solution in viscosity sense of the system \eqref{ourPDE} when\\
	$f_i\big(x,v^1(x),...,v^m(x),\sigma^\top(x).D_xv^i(x)\big):=f_i(x,v^i(x),\sigma^\top(x).D_x v^i(x))$, and \\ $\max\limits_{j\in{\I}^{-i}}\{-g_{ij}(x)+v^j(x)\}:=\phi(x)$ where $\phi$ is a continuous and bounded function (see Appendix \ref{1st-extra-result}).
	%\min\biggr\{v^i(x)- &\max\limits_{j\in{\I}^{-i}}\{-g_{ij}(x)+v^j(x)\},rv^i(x)-\cL v^i(x)\\ &-f_i\big(x,v^1(x),...,v^m(x),\sigma^\top(x).D_xv^i(x)\big)\biggl\}=0,
	
	%\be\label{PDE_appendix}
	%		  \ba
	%		  \min\biggr\{\bar{v}^i(x)- &\phi(x),r\bar{v}^i(x)-\cL \bar{v}^i(x)\\ &-\bar{f}_i(x,\bar{v}^i(x),\sigma^\top(x).D_x\bar{v}^i(x))\biggl\}=0,
	%		  \ea
	%		  \ee
	
	\item[$\bullet$] A comparison result for infinite horizon RBSDEs when the generator and the barrier are of the following form $f(x,y,z)$ and $S_t$, $\forall t\geq0$ (see Appendix \ref{2nd-extra-result}).
\end{itemize}
The main result of the paper is obtained under the following assumptions:
\begin{itemize}
	\item[(i)] the switching costs are only supposed to be non-negative and satisfy the {\it non free
		loop condition} (see \bf{[H3]}-(ii) below);
	\item[(ii)] the generators $(f_i)_{i\in\I}$ are merely Lipschitz w.r.t. $(y^1,\ldots,y^m,z^i)$;
	\item[(iii)] for any $i\in\I$, %$f_i$ depends on $(y^1,\dots,y^m ,z^i)$ and satisfy:
	$\forall j\in\I^{-i}$ the mappings $y^j\mapsto f_i(x,y^1,\dots,y^{j-1},y^j,y^{j+1},\dots,y^m,z)$ is non-decreasing.
\end{itemize}
Note that the {\it non free loop condition} in $(i)$, which was introduced in \cite{IK}, is classic in the theory of OSP. Actually, this condition excludes the possibility that one can gain a positive profit by a looping switching
strategy at the same time. If this latter is not postulated, then the value function diverges as the decision maker obtains an infinitely large reward by such a looping strategy, which will represent an arbitrage.\\
The hypothesis $(iii)$ is assumed because it will allow us to obtain the comparison result of sub-solutions and super-solutions
of system \eqref{ourPDE} which plays a primordial role in the proof of the main result. Moreover, without assuming it, there will be a lack of regularity of the
functions $(v^i)_{i\in\I}$, which are obtained from system of RBSDEs \eqref{RBSDE-Markov-framework} via Feynman-Kac's formula \eqref{Intro-Feynman-Kac}.\\
It should be noted that the hypotheses $(i)-(iii)$ were also assumed in the finite horizon framework in \cite{HM}.
%we do not assume extra assumptions in the infinite horizon framework compared to the finite horizon one (see for example \cite{HM}).

Let us close this introduction by shedding some light on the main contributions of our paper in
comparison with the existing works in the literature especially \cite{Elasri,HM} which are the most relevant to the present paper.
On the one hand our results extend the ones of \cite{HM} to the infinite horizon framework without making any extra assumptions compared to theirs. On the second hand we generalize the work of \cite{Elasri} (see system \eqref{ELasri-IH-PDE}) to the case when $f_i(x):=f_i(x,y^1,\ldots,y^m,z^i)$ which makes the problem difficult to deal with, and without assuming the boundedness condition on the switching costs. Moreover,
%our methods are different
since we assume that the switching costs only satisfy the {\it non free loop condition} new difficulties arise. Indeed, it becomes impossible to use methods based on the estimate of number of switchings of the optimal strategy as done in \cite{Elasri}.\\
It should be mentioned that there are other motivations for considering this paper, beyond the objective of unification of the finite and infinite horizon frameworks. One of them is the extension of the Feynman-Kac's formula \eqref{Intro-Feynman-Kac} which could be very helpful in the numerical study of RBSDEs \eqref{RBSDE-Markov-framework}. Last but not least, the present paper is the first step toward the study of infinite horizon RBSDEs with interconnected bilateral obstacles in a Markovian framework and related PDEs with bilateral interconnected obstacles, which are connected with zero-sum stochastic switching games.

The rest of this paper is organized as follows. In Section $2$ we introduce the problem, give
the assumptions, and then we define the notion of a viscosity solution for the system \eqref{ourPDE}. Section $3$ is devoted to the relationship between $(v^i)_{i\in\I}$ solution of $\eqref{ourPDE}$ and $(Y^i)_{i\in\I}$ solution of $\eqref{RBSDE-Markov-framework}$, in the Markovian setting. In section $4$ we prove the uniqueness result for the solution of system $\eqref{ourPDE}$, by establishing the classical comparison result. In Section $5$, we show that the solution of $\eqref{ourPDE}$ exists and is unique in the class of continuous functions which satisfy a polynomial growth condition. Finally, in the Appendix \ref{1st-extra-result}, we prove the result of existence and uniqueness in viscosity sense for $\eqref{ourPDE}$ when the generator depend only on $v^i$. In Appendix \ref{2nd-extra-result} we state and prove a comparison theorem for infinite horizon RBSDE.
\section{Assumptions and problem formulation}
Let $k$ be a fixed integer positive constant, let $\I:=\{1,\dots,m\}$ and let us consider the following functions: for $i,j\in\I$
$$
\begin{array}{l}
b(x):\R^k\mapsto  \R^{k};\\
\sigma(x):\R^k\mapsto \R^{k\times
	d};\\ f_i(x,y^1,\dots,y^m,z):\R^k\times\R^m\times\R^d\mapsto \R\,;\\
g_{ij}(x):\R^k\mapsto \R\,(i\neq j);\\
h_i(x):\R^k\mapsto \R.\\
\end{array}$$
Next, we denote $\vec{y}=(y^1,...,y^m)$.\\
\par We now consider the following assumptions:\\
\bf{[H1]}: The functions $b$ and $\sigma$ are  of linear growth in $x$ and Lipschitz continuous w.r.t $x$, i.e. there exists
a constant $C \geq 0$ such that for any $x, x^\prime\in \R^{k}$ we have
\be
|b(x)-b(x^\prime)|+|\sigma(x)-\sigma(x^\prime)| \leq C|x-x^\prime|, \label{Hsigmak1}
\ee
and
\be
|b(x)|+|\sigma(x)| \leq C(1+|x|).\label{H b+sigma}
\ee
Throughout this paper we assume that assumption $\bf{[H1]}$ holds and $C$ will denote a constant whose value can vary from line to line.\\
\bf{[H2]}: Each function $f_i$
\begin{itemize}
	\item[$(i)$] is continuous in $x$ uniformly w.r.t. $(\vec{y},z)$ and, for any $x$, the mapping $x\mapsto f_i(x,0,\dots,0)$ is Borel and of polynomial growth, i.e., there exist a constant $C$ and $\gamma$ such that
	\be
	|f_i(x,0,\dots,0)|\leq C(1+|x|^{\gamma}).\label{f_of_pg}
	\ee
	We denote by $\Pi^g$ the class of functions with polynomial growth.
	\item[$(ii)$] is Lipschitz continuous in $(\vec{y},z)$ uniformly in $x$, i.e. there exists a constant $C$,
	$\forall(\vec{y}, z)\in \R^{m+d}$ and $\forall(\vec{y}^{\,\prime}, z^\prime)\in \R^{m+d}$ we have
	\be
	|f_i(x, \vec{y}, z)-f_i(x, \vec{y}^{\,\prime}, z^\prime)| \leq C\left(|\vec{y}-\vec{y}^{\,\prime}|+|z-z^\prime| \right). \label{Hfc}
	\ee
	%\item[$(iii)$] \it{Monotonicity}: for any $i \in\I$, and $j\in\I^{-i}:=\I-\{i\}$ the mapping $y^j\mapsto f_i(x,\vec{y},z)$ is non-decreasing, i.e. for any $(x,\vec{y},\vec{y}^{\,\prime},z)\in \R^k\times[\R^{m}]^2 \times \R^d$ such that $y^i=y'^i$ and $y^j\leq y'^j$ for $j\neq i$, we have
	%$$f_i(x,\vec{y},z)\leq f_i(x,\vec{y}^{\,\prime},z) \quad \mbox{P}-a.s$$
	\item[$(iii)$] \it{Monotonicity}: for any $j\in\I^{-i}:=\I-\{i\}$ the mapping $y^j\mapsto f_i(x,y^1,\dots,y^{j-1},y^j,y^{j+1},\dots,y^m,z)$ is non-decreasing whenever the other components
	$(x,y^1,\dots,y^{j-1},y^{j+1},\dots,y^m,z)$ are fixed.
\end{itemize}
\bf{[H3]}:$(i)$ For any $i,j\in\I$, $g_{ij}$ is continuous in $x$, non-negative, i.e. $\forall x\in\R^k\quad g_{ij}(x)\geq 0$ and of polynomial growth.
\\ $(ii)$ The non-free loop property is satisfied, i.e. for any $x\in\R^k$ and for any sequence of indices $i_1,\dots,i_k$ such that $i_1\neq i_2$,
$i_1=i_k$ and $card\{i_1,\dots,i_k\}=k-1$ we have: $$g_{i_1i_2}(x)+g_{i_2i_3}(x)+\dots+g_{i_{k-1}i_k}(x)>0,\qquad \forall x\in\R^k.$$
By convention we set $g_{ii}=0,\; \forall i\in\I.$\\
$(iii)$ For any $(i,j)\in\I^2$, $g_{ij}$ belongs to $\mathcal{C}^{2}(\R^{k})$ and $$\cL g_{ij}(x)\leq0, \mbox{ for all } x\in\R^{k};$$
where $\cL$ is the following infinitesimal generator:
\be
\label{derivgen} \cL=\frac{1}{2}\sum_{i,j=1}^{k}(\sigma
\sigma^{\top})_{ij}(x)\frac{\partial^{2}}{\partial x_{i} \partial
	x_{j}}+\sum_{i=1}^{k}b_{i}(x) \frac{\partial }{\partial x_{i}}.
\ee
\bf{[H4]}: For any $i\in\mathcal{I}$ the function $h_i(x):=\underset{t \to+\infty}{\lim}Y^{i,x}_t$ belongs to $\Pi^g$ and satisfies $$h_i(x)\geq \underset{j\in\mathcal{I}^{-i}}{\max}(h_j(x)-g_{ij}(x_\infty));$$ where $g_{ij}(x_\infty):=\underset{t \to+\infty}{\lim}g_{ij}(x_t)$.
\par In this paper we are concerned with the existence and uniqueness in viscosity sense of the solution $(v^1,\dots,v^m): x\in
\R^k\mapsto (v^1(x),\dots,v^m(x))\in \R^{m} $ of the
following system of $m$ partial differential equations with
interconnected obstacles: for $i\in\I:=\{1,\dots,m\}$ and $x\in\R^k$,
\be\label{PDE}
\ba
\min\biggr\{v^i(x)- &\max\limits_{j\in{\I}^{-i}}(-g_{ij}(x)+v^j(x)),rv^i(x)-\cL v^i(x)\\ &-f_i\big(x,v^1(x),...,v^m(x),\sigma^\top(x).D_xv^i(x)\big)\biggl\}=0,
\ea
\ee
where $r$ is a positive discount factor and $(.)^{\top}$ stands for the transpose.\\
To proceed we first define the notions of superjets and subjets of a continuous function $v$.
\begin{definition} (Subjets and superjets)\\\ms
	Let $v \in C(\R^k;\R)$, $x$ an element of $ \R^k$ and $\Sm_k$ the
	set of $k \times k$ symmetric matrices. We denote by $J^{2,+} v(x)$
	(resp. $J^{2,-} v(x)$), the superjets (resp. the subjets) of $v$ at
	$x$, the set of pairs $(q,M)\in \R^k \times \Sm_k$ such that:
	$$\begin{array}{c}
	v(y)\leq v(x) +\langle q,y-x \rangle +\frac{1}{2}\langle
	M(y-x),y-x\rangle+o(|y-x|^2) \\
	(resp.\quad v(y)\geq v(x) +\langle q,y-x \rangle +\frac{1}{2}\langle
	M(y-x),y-x\rangle+o(|y-x|^2)).
	\end{array}$$
\end{definition}

We now give the definition of a viscosity solution for the system of
PDEs (\ref{PDE}).

\begin{definition} (Viscosity solution to (\ref{PDE}))\\
	$(v^1,...,v^m) \in C(\R^k;\R^m)$ is called a
	viscosity supersolution (resp. subsolution) of (\ref{PDE}) if for
	any $i\in {\I}$, $x\in \R^k$ and $(q,X)\in J^{2,-} v^i (x)$
	$(resp.\ J^{2,+} v^i(x))$,
	\begin{equation}
	\begin{aligned}
	\min\biggr\{v^i(x)- &\max\limits_{j\in{\I}^{-i}}(-g_{ij}(x)+v^j(x)),rv^i(x)-\langle b(x),q\rangle\\ &-\frac{1}{2}
	Tr[(\sigma\sigma^\top)(x).X]-f_i(x,v^1(x),\dots,v^m(x),\sigma
	(x)^\top.q)\biggl\}\geq 0 \,\,(resp. \leq 0).
	\end{aligned}
	\end{equation}
	$(v^1,...,v^m)$ is called a viscosity solution of (\ref{PDE}) if it is both a viscosity subsolution and supersolution of (\ref{PDE}).\qed
\end{definition}
Let $x\in \R^k$ and let $\{X^{x}_t; t\geq 0\}$ be the solution of the following
standard SDE:
\begin{equation}\label{sde}
dX_t^x=b(X^x_t)dt+\sigma(X^x_t)dB_t, \quad X^x_0=x,
\end{equation}
where the functions $b$ and $\sigma$ are the ones of
$\bf{[H1]}$. These properties of $\sigma$ and $b$ imply in particular
that $X^{x}_t$ solution of the standard SDE (\ref{sde}) exists and is
unique in $\R^k $.

In the following result we collect some properties of
$X^{x}_t$.

\begin{proposition} \label{estimx}(see e.g. \cite{RY})  The process $X^{x}$
	satisfies the following estimates:
	\begin{itemize}
		\item [$(i)$] For any $q\geq 2$ there exists $C_q$ such that,
		\begin{equation}\label{estimat1}
		E[|X^{x}_t|^q]\leq C_qe^{C_qt}(1+|x|^q)\quad \forall t\geq 0.
		\end{equation}
		\item[$(ii)$] There exists a constant $C$ such that for any  $x,x'\in \R^k$ and $T\geq 0$,
		\be\label{estimat3}
		E[\sup\limits_{0\leq s \leq  T}|X^{x}_s-X^{x'}_s|^2]\leq
		Ce^{CT}|x-x'|^2.
		\ee
	\end{itemize}
\end{proposition}
In the sequel we consider the following condition:\\
\bf{[H5]}:Assume $q\geq 2$ and
\be\label{r-c}
C_q-r < 0,
\ee
where $q$ is the growth exponent of the functions $f_i$ and $C_q$ is the constant in $\eqref{estimat1}$.\\

In order to tackle the problem of existence of a solution for the system of PDEs (\ref{PDE}), we first make the link between that solution and solution of systems of RBSDEs with oblique reflection which we will introduce in the next section and discuss its existence and uniqueness.
%\section{Systems of RBSDEs with Oblique Reflection}
\section{RBSDEs and their relationship with PDEs}
Let $(\Omega,\mathcal{F},\P)$ be a fixed probability space endowed with a $d$-dimensional Brownian motion $B:=(B_t)_{t\geq0}$. $\{\mathcal{F}_t,\, t\geq0\}$ is the natural filtration of the Brownian motion augmented by $P$-null sets of $\mathcal{F}$, and $\F_{\infty}=\bigvee_{t\geq0}\F_t$. All the measurability notion will refer to this filtration. Let $|.|$ denote the Euclidean norm for vectors.\\
\par For $r\in\R^+$, let us introduce the following spaces:
\begin{itemize}
	\item[-]$\M_r^{2}$  is the set of $\R^{d}$-valued, progressively measurable processes $(Z_t)_{t\geq0}$ such that $$\mathbb{E}\left[\int^{+\infty}_{0}e^{-rs}|Z_s|^2ds\right]<+\infty.$$
	\item[-]$\S_r^{2}$ is the set of $\mathbb{R}$-valued adapted and  continuous processes $(Y_t)_{t\geq0}$ such that $$\mathbb{E}\left[\sup_{t\geq0}e^{-rt}|Y_{t}|^2\right]<+\infty.$$
	\item[-]$\K_r^2$  is the subset of non-decreasing processes $(K_t)_{t\geq 0}\in\S_r^{2}$, starting from
	$K_0=0$.	
	%\item[-] $\L^2$  is the set of $\F_{\infty}$-measurable random variable $\xi$ satisfying $\E[|\xi|^2]<+\infty$.
\end{itemize}
Next, let us introduce the solution of the system of RBSDEs
with oblique reflection associated with the deterministic functions
$((f_i)_{i\in\I},(g_{ij})_{i,j\in\I})$ introduced above. The solution consists of $m$ triplets of processes $((Y^{i,x},Z^{i,x},K^{i,x}))_{i\in\I}$, and satisfies:   for any $i\in\I$,
\begin{equation}\label{1RBSDE}
\begin{cases}
Y^{i,x}\in\S_r^2,Z^{i,x}\in\M_r^2 \mbox{ and } K^{i,x}\in\K_r^2;\\
\begin{aligned}
e^{-r t}Y^{i,x}_t = \int_{t}^{+\infty} &e^{-r s}f_i(X^{x}_s,
Y^{1,x}_s,\dots,Y^{m,x}_s,Z^{i,x}_s)ds+ K^{i,x}_\infty-K^{i,x}_t\\ &-  \int_{t}^{+\infty}e^{-r s}Z^{i,x}_sdB_s;
\end{aligned}\\
\lim\limits_{t \to+\infty} e^{-r t}Y^{i,x}_t=0,\\
\forall t\geq0, \quad e^{-r t}Y^{i,x}_t \geq e^{-r t} \underset{j\in\I^{-i}}{\max}(Y^{j,x}_t-g_{ij}(X^{x}_t)),\\ \mbox{ and }
\int_{0}^{+\infty} e^{-r s}\{Y^{i}_s - \underset{j\in\I^{-i}}{\max}(Y^{j,x}_s-g_{ij}(X^{x}_s))\}dK^{i,x}_s = 0.
\end{cases}
\end{equation}
This system of RBSDEs with interconnected obstacles $\eqref{1RBSDE}$ has been considered by El asri and Ourkiya in $\cite{EO}$ where issues of existence and uniqueness of the solution in the non-Markovian framework.

First let us show that the system $\eqref{1RBSDE}$ has a solution.
\begin{theorem}\label{therem31}
	Assume that \bf{[H2]}, \bf{[H3]} and \bf{[H5]} hold. Then the system $\eqref{1RBSDE}$ has a solution $((Y^{i,x},Z^{i,x}$ $,K^{i,x}))_{i\in\I}$. Moreover, there exists deterministic lower semi-continuous functions $(v^i)_{i\in \I}$ of polynomial growth, defined on $\R^k$, $\R$-valued, such that:
	\be
	\forall i\in\I,\; \forall x \in\R^k, \quad Y^{i,x}_0=v^i(x).
	\ee
\end{theorem}
\begin{proof}
	First, consider the following BSDEs:
	\be\label{solmax}
	\left\{\begin{array}{l}
		\bar{Y}^x\in \S_r^2, \bar{Z}^x\in \M_r^2\\
		e^{-r t}\bar{Y}^x_t=\int_t^{+\infty} e^{-r s}\underset{i\in\mathcal{I}}{\max}\;f_i(X^{x}_s,\bar{Y}^x_s,\dots, \bar{Y}^x_s,\bar{Z}^x_s)ds-\int_t^{+\infty} e^{-rs}\bar
		{Z}^x_sdB_s,
	\end{array}
	\right.
	\ee
	and
	\be\label{solmin}
	\left\{\begin{array}{l}
		\underline{Y}^x\in \S_r^2, \underline{Z}^x\in \M_r^2\\
		e^{-r t}\underline{Y}^x_t=\int_t^{+\infty} e^{-r s}\underset{i\in\mathcal{I}}{\min}\;f_i(X^{x}_s,\underline{Y}^x_s,\dots, \underline{Y}^x_s,\underline
		{Z}^x_s)ds-\int_t^{+\infty} e^{-r s}\underline{Z}^x_sdB_s.
	\end{array}
	\right.
	\ee
	Both solutions of (\ref{solmax}) and (\ref{solmin}) exist and are unique, thanks to Theorem 4.1 in \cite{P99}.
	
	Next, we introduce the following sequences of RBSDEs defined recursively by: for any $i\in\I$,
	$Y^{i,x,0}=\underline{Y}^x$ and for $n\geq 1$ and $t\geq0$,
	\be\label{systemappro}
	\left\{\begin{array}{l} Y^{i,x,n}\in
		\S_r^2, \,\,Z^{i,x,n} \in \M_r^{2} \mbox{ and }K^{i,x,n}\in\K_r^2;\\
		e^{-r t}Y^{i,x,n}_t=\int_t^{+\infty}  e^{-r s}f_i(X^x_s,
		Y^{1,x,n-1}_s,\dots,Y^{i-1,x,n-1}_s,Y^{i,x,n}_s,Y^{i+1,x,n-1}_s,\dots,Y^{m,x,n-1}_s,Z^{i,x,n}_s)ds\\ \qquad\qquad\qquad+
		K^{i,x,n}_\infty-K^{i,x,n}_t-\int_t^{+\infty}  e^{-r s}Z^{i,x,n}_sdB_s;\\
		\lim\limits_{t \to+\infty} e^{-r t}Y^{i,x,n}_t=0,\\
		e^{-r t}Y^{i,x,n}_t\geq  e^{-r t}\underset{j\in\I^{-i}}{\max}\{Y^{j,x,n-1}_t-g_{ij}(X^x_t)\},\\ \mbox{ and }
		\int_0^{+\infty} e^{-r s}(Y^{i,x,n}_s-\underset{j\in\I^{-i}}{\max}\{Y^{j,x,n-1}_s-g_{ij}(X^x_s)\})dK^{i,x,n}_s=0.\end{array}\right.\ee
	The main idea is to show that solutions $(Y^{i,x,n},Z^{i,x,n},K^{i,x,n})_{i\in \I}$ converge to a limit that has the form of RBSDE (\ref{1RBSDE}). Moreover we know that if such limit exists, then it is unique thanks to Theorem $4.1$ in $\cite{EO}$.
	
	By an induction argument we show that there exists a unique m-uplet of processes $(Y^{i,x,n},$ $Z^{i,x,n},K^{i,x,n}),$ $i\in \I $: For $n=0$, it is true since $Y^{i,x,0}=\underline{Y}^x$ exists and is unique by the Theorem $4.1$ in \cite{P99}. Next the claim is also holds for $n\geq1$ thanks to result of Hamad\`ene et al. ($\cite{HLW}$, Theorem $3.2.$). \\
	Now, we will show by induction that the sequence $(Y^{i,x,n})_{n\geq1}$ is non-decreasing, for any $i \in \I$.
	For any $i \in \I$: for $n=0$, we have that
	\be\nn\ba
	&\quad\underset{i\in\mathcal{I}}{\min}\;f_i(X^x_s,Y^{1,x,0}_s,\dots,Y^{i-1,x,0}_s, Y^{i,x,1}_s, Y^{i+1,x,0}_s,\dots,Y^{m,x,0}_s,z)\\ &
	\leq\; f_i(X^x_s,Y^{1,x,0}_s,\dots,Y^{i-1,x,0}_s, Y^{i,x,1}_s, Y^{i+1,x,0}_s,\dots,Y^{m,x,0}_s,z).
	\ea\ee
	%$$ \underset{i\in\mathcal{I}}{\min}\;f_i(X^x_s,Y^{1,0,x}_s,\dots,Y^{m,0,x}_s,z)\leq\; f_i(X^x_s,Y^{1,0,x}_s,\dots,Y^{m,0,x}_s,z).$$
	Then using a comparison theorem in (\cite{HLW}, Theorem $2.2.$), we obtain that $Y^{i,x,0}\leq Y^{i,x,1}$.\\
	On the other hand, $f_i$ satisfies the monotonicity property
	$\bf{[H2]}-(iii)$ and using once more the comparison of solutions of RBSDEs (see e.g. Theorem $\ref{comparison_theo}$ in Appendix B) we obtain by induction that:
	$$Y^{i,x,n}\leq Y^{i,x,n+1}, \quad \forall i \in \I.$$
	As a by product the sequence $(Y^{i,x,n})_{n\geq1}$ is increasing for any $i\in \I$.\\
	Now, simply using an induction procedure and the repeated use of the comparison theorem in the same way we did above, leads to
	\be
	\forall n\geq1, \forall i\in \I, \qquad Y^{i,x,n}\leq Y^{i,x,n+1}\leq \bar{Y}^x.
	\ee
	So
	\be\label{estimateyinx}
	\E[\sup_{t\geq0}|Y_t^{i,x,n}|^2]\leq C,
	\ee
	Then $Y^{i,x,n}$ admits $\P$-a.s. a limit denoted $Y^{i,x}$.
	Moreover by Fatou's lemma, we have
	\be\nn
	\E[ \sup_{t\geq0}|Y_t^{i,x}|^2]\le C.
	\ee
	Then applying Lebesgue's dominated converge theorem, we get
	\be
	\lim_{n \rightarrow +\infty}\E\biggl[\int_0^{+\infty}|Y_s^{i,x,n}-Y_s^{i,x}|^2 \, ds\biggr]=0.
	\ee
	Next we will show that the sequence $Y^{i,x,n}$ is of Cauchy type in $\S_r^2$.\\
	Let fix $\vec{\Gamma}:=(\Gamma^{i,x},...,\Gamma^{i,x})$ in $[\S_r^2]^m$ and introduce the operator $\phi:[\S_r^2]^m\mapsto[\S_r^2]^m, \vec{\Gamma}\mapsto\vec{Y}:=\phi(\vec{\Gamma})$, where $(\vec{Y},\vec{Z},\vec{K}):=(Y^{i,x,n},Z^{i,x,n},K^{i,x,n})_{i\in\I}\in[\S_r^2\times\M_r^2\times\K_r^2]^m$ is the solution to the following RBSDE, $\forall i\in\I$ and $\forall t\geq0$,
	\be\label{RBSDE_contra}
	\left\{\begin{array}{l}
		e^{-r t}Y^{i,x,n}_t=\int_t^{+\infty}  e^{-r s}f_i(X^x_s,
		\vec{\Gamma}_s,Z^{i,x,n}_s)ds+
		K^{i,x,n}_\infty-K^{i,x,n}_t\\ \qquad\qquad\qquad-\int_t^{+\infty}  e^{-r s}Z^{i,x,n}_sdB_s;\\
		\lim\limits_{t \to+\infty} e^{-r t}Y^{i,x,n}_t=0,\\
		e^{-r t}Y^{i,x,n}_t\geq  e^{-r t}\underset{j\in\I^{-i}}{\max}\{Y^{j,x,n-1}_t-g_{ij}(X^x_t)\},\\  \mbox{ and }
		\int_0^{+\infty} e^{-r s}(Y^{i,x,n}_s-\underset{j\in\I^{-i}}{\max}\{Y^{j,x,n-1}_s-g_{ij}(X^x_s)\})dK^{i,x,n}_s=0.\end{array}\right.\ee
	which exists and is unique thanks to Theorem $3.2$ in $\cite{HLW}$. Then $\phi$ is well defined and is obviously valued in $[\S_r^2]^m$.
	\par The operator $\phi$ is a contraction on $[\S_r^2]^m$ by the following
	\begin{proposition}\label{prop_contra} (see e.g. $\cite{EO}$, Proposition 4.1)
		The operator $\phi$ is a contraction on the Banach space $[\S_r^2]^m$ endowed with the norm $\|.\|_{2,r}$ defined by:
		$$ \|Y\|_{2,r}:=\E\left[\underset{t\geq0}{\sup}\,e^{-r t}|Y_t|^2\right]^{\frac{1}{2}}.$$
	\end{proposition}
	We know that by The previous Proposition that the mapping $\phi$ is a contraction in $(\S_r^2)^m$. Moreover, since, we have that:
	\be
	\forall n\geq1, \forall i \in \I, Y^{i,x,n+1}=\phi(Y^{i,x,n}).
	\ee
	Then, the claim holds true, and using again Proposition $\ref{prop_contra}$ we have the following:
	\be\label{Yinx-cauchy}
	\lim_{n,p \rightarrow +\infty}\E\biggl[\sup_{s\geq0}|Y_s^{i,x,n}-Y_s^{i,x,p}|^2\biggr]=0,
	\ee
	and
	\be\label{uniform_yinx}
	\lim_{n \rightarrow +\infty}\E\biggl[\sup_{s\geq0}|Y_s^{i,x,n}-Y_s^{i,x}|^2\biggr]=0.
	\ee
	Applying It\^o's formula with $e^{-r t}|Y_t^{i,x,n}|^2$  we obtain that for any $t\geq0$
	\begin{equation}\label{ito_exp_y}
	\begin{aligned}
	&e^{-r t}|Y_t^{i,x,n}|^2+\int_{t}^{+\infty}e^{-r s}(r|Y_s^{i,x,n}|^2+|Z_s^{i,x,n}|^2)ds\\ & = 2\int_{t}^{+\infty}e^{-r s}Y_s^{i,x,n}f_i(X^x_s,
	Y^{1,x,n-1}_s,\dots,Y^{i-1,x,n-1}_s,Y^{i,x,n}_s,Y^{i+1,x,n-1}_s,\dots,Y^{m,x,n-1}_s,Z^{i,x,n}_s)ds \\ & \quad+2\int_{t}^{+\infty}Y_s^{i,x,n}dK_s^{i,x,n} -2\int_{t}^{+\infty}e^{-r s}Y_s^{i,x,n}Z_s^{i,x,n}dB_s,
	\end{aligned}
	\end{equation}
	since, \be\label{dyin}\ba dY^{i,x,n}_s&=-[f_i(X^x_s,
	Y^{1,x,n-1}_s,\dots,Y^{i-1,x,n-1}_s,Y^{i,x,n}_s,Y^{i+1,x,n-1}_s,\dots,Y^{m,x,n-1}_s,Z^{i,x,n}_s)-r Y^{i,x,n}_s]ds\\ &\qquad-
	e^{r s}dK^{i,x,n}_s+Z^{i,x,n}_sdB_s.\ea\ee
	Then, taking expectation and from $\bf{[H2]}-(ii)$ we get
	\begin{equation*}
	\begin{aligned}
	&\E[e^{-r t}|Y_t^{i,x,n}|^2]+\E\biggl[\int_{t}^{+\infty}e^{-r s}[r|Y_s^{i,x,n}|^2+|Z_s^{i,x,n}|^2]ds\biggr]\\ & \leq 2\E\biggl[\int_{t}^{+\infty}e^{-r s}|Y_s^{i,x,n}|\big[|f_i(X^x_s,0,\dots,0)|+mC|Y^{i,x,n}_s|+C|Z^{i,x,n}_s|\big]ds\biggr] \\ & \quad+2\E\biggl[\int_{t}^{+\infty}Y_s^{i,x,n}dK_s^{i,x,n}\biggr],\\ & \leq \E\biggl[C\int_{t}^{+\infty}e^{-r s}\big(|Y_s^{i,x,n}|^2+|f_i(X^x_s,0,\dots,0)|^2+\frac{1}{2}|Z^{i,x,n}_s|^2\big)ds\biggr] \\ & \qquad+2\E\biggl[\sup_{t\geq0}|Y_t^{i,x,n}|\int_{t}^{+\infty}dK_s^{i,x,n}\biggr].
	\end{aligned}
	\end{equation*}
	For $r \geq  C$ and setting $t=0$, we obtain
	\be\label{estimatezinx}\E\biggl[\int_{0}^{+\infty}e^{-r s}|Z_s^{i,x,n}|^2ds\biggr]\leq 2\E\biggl[\int_{0}^{+\infty}e^{-r s}|f_i(X^x_s,0,\dots,0)|^2ds+2\sup_{t\geq0}|Y_t^{i,x,n}|\,K_{+\infty}^{i,x,n}\biggr].
	\ee
	We now give an estimate to $E[({K^{i,x,n}_\infty})^2]$. From (\ref{systemappro}) we have
	\be\nn\ba
	K^{i,x,n}_\infty=\,&-\int_0^\infty e^{-r s}f_i(X^x_s,
	Y^{1,x,n-1}_s,\dots,Y^{i-1,x,n-1}_s,Y^{i,x,n}_s,Y^{i+1,x,n-1}_s,\dots,Y^{m,x,n-1}_s,Z^{i,x,n}_s)ds\\ &
	+Y^{i,x,n}_0-\int_{0}^{+\infty}e^{-r s}Z^{i,x,n}_sdB_s,
	\ea\ee
	and estimates $\eqref{estimateyinx}$ and $\eqref{estimatezinx}$, we show the following inequalities:
	\be\nn\ba E[({K^{i,x,n}_\infty})^2]&\leq \;\E\biggl[|Y_0^{i,x,n}|^2+\frac{2}{r}\int_{0}^{+\infty}e^{-r s}[|f_i(X^x_s,0,\dots,0)|^2+m^2C^2|Y_s^{i,x,n}|^2\\ & \qquad +C^2|Z^{i,x,n}_s|^2]ds+\int_{0}^{+\infty}e^{-2r s}|Z_s^{i,x,n}|^2ds\biggr],\\ &\leq C\E\biggl[1+\int_{0}^{+\infty}e^{-r s}|f_i(X^x_s,0,\dots,0)|^2ds+2\sup_{t\geq0}|Y_t^{i,x,n}|\,K_{+\infty}^{i,x,n}\biggr],\\ &\leq C\E\biggl[1+\int_{0}^{+\infty}e^{-r s}|f_i(X^x_s,0,\dots,0)|^2ds+\frac{1}{\epsilon}\sup_{t\geq0}|Y_t^{i,x,n}|^2+\epsilon({K^{i,x,n}_\infty})^2\biggr].
	\ea\ee
	It follows that, for $\epsilon=\frac{1}{2C}$,
	\be\label{estimatekinx} E[({K^{i,x,n}_\infty})^2]\leq C\E\biggl[1+\int_{0}^{+\infty}e^{-r s}|f_i(X^x_s,0,\dots,0)|^2ds+\sup_{t\geq0}|Y_t^{i,x,n}|^2\biggr]\ee
	As a by product, combining estimates (\ref{estimateyinx}), (\ref{estimatezinx}) and (\ref{estimatekinx}), by the polynomial growth of $f_i$ (see $\eqref{f_of_pg}$) and $\bf{[H5]}$  we conclude that
	\be\label{estimatemarkovYZK}
	\E\biggl[\sup_{t\geq0}|Y_t^{i,x,n}|^2+\int_{0}^{+\infty}e^{-r s}|Z^{i,x,n}_s|^2ds+({K^{i,x,n}_\infty})^2\biggr]\leq C.
	\ee
	
	Now we prove that $Z^{i,x,n}$ is a Cauchy sequence in $\M_r^2$. To do so, we apply it\^{o}'s formula to \\$e^{-rt}|Y_t^{i,x,n}-Y^{i,x,p}_t|^2$ and recalling $(\ref{dyin})$ to obtain
	\begin{equation}\label{ynyp2}
	\begin{aligned}
	&e^{-rt}|Y_t^{i,x,n}-Y^{i,x,p}_t|^2+\int_{t}^{+\infty}e^{-rs}\{r |Y_s^{i,x,n}-Y^{i,x,p}_s|^2+|Z_s^{i,x,n}-Z^{i,x,p}_s|^2\}ds\\ &= 2\int_t^{+\infty}e^{-rs}(Y_s^{i,x,n}-Y^{i,x,p}_s)\\ &\quad \times\{f_i(X^x_s,Y^{1,x,n-1}_s,\dots,Y^{i-1,x,n-1}_s,Y^{i,x,n}_s,Y^{i+1,x,n-1}_s,\dots,Y^{m,x,n-1}_s,Z^{i,x,n}_s)\\ &
	\qquad -f_i(X^x_s,
	Y^{1,x,p-1}_s,\dots,Y^{i-1,x,p-1}_s,Y^{i,x,p}_s,Y^{i+1,x,p-1}_s,\dots,Y^{m,x,p-1}_s,Z^{i,x,p}_s)\}\, ds \\ &\quad +2\int_{t}^{+\infty}(Y_s^{i,x,n}-Y^{i,x,p}_s)(dK_s^{i,x,n}-dK_s^{i,x,p})\\&\quad-2\int_{t}^{+\infty}e^{-rs}(Y_s^{i,x,n}-Y^{i,x,p}_s)(Z^{i,x,n}_s-Z^{i,x,p}_s)\, dB_s,
	\end{aligned}
	\end{equation}
	and after taking expectations we get
	\begin{equation}\label{ynyp2}
	\begin{aligned}
	&\E\biggr[e^{-r t}|Y_t^{i,x,n}-Y^{i,x,p}_t|^2\biggl]+\E\biggr[\int_{t}^{+\infty}e^{-r s}(r |Y_s^{i,x,n}-Y^{i,x,p}_s|^2+|Z_s^{i,x,n}-Z^{i,x,p}_s|^2)ds\biggl]\\ &= 2\E\biggr[\int_t^{+\infty}e^{-r s}(Y_s^{i,x,n}-Y^{i,x,p}_s)\\&\quad
	\times\{f_i(X^x_s,Y^{1,x,n-1}_s,\dots,Y^{i-1,x,n-1}_s,Y^{i,x,n}_s,Y^{i+1,x,n-1}_s,\dots,Y^{m,x,n-1}_s,Z^{i,x,n}_s)\\ &
	\qquad -f_i(X^x_s,
	Y^{1,x,p-1}_s,\dots,Y^{i-1,x,p-1}_s,Y^{i,x,p}_s,Y^{i+1,x,p-1}_s,\dots,Y^{m,x,p-1}_s,Z^{i,x,p}_s)\}\, ds\biggl] \\ &\quad +2\E\biggr[\int_{t}^{+\infty}(Y_s^{i,x,n}-Y^{i,x,p}_s)(dK_s^{i,x,n}-dK_s^{i,x,p})\biggl], \\ &\leq \E\biggr[\int_t^{+\infty}e^{-r s}\big(C|Y_s^{i,x,n}-Y^{i,x,p}_s|^2+\frac{1}{2}|Z_s^{i,x,n}-Z^{i,x,p}_s|^2\big)\, ds\biggl] \\ &\quad +2\E\biggl[\sup_{t\geq0}|Y_t^{i,x,n}-Y_t^{i,x,p}|^2\biggr]^{\frac{1}{2}}\E\biggl[|K_{\infty}^{i,x,n}-K_t^{i,x,n}|^2\biggr]^{\frac{1}{2}}\\ &\quad +2\E\biggl[\sup_{t\geq0}|Y_t^{i,x,n}-Y_t^{i,x,p}|^2\biggr]^{\frac{1}{2}}\E\biggl[|K_{\infty}^{i,x,p}-K_t^{i,x,p}|^2\biggr]^{\frac{1}{2}} .
	\end{aligned}
	\end{equation}
	Setting $t=0,$ and choosing $r \geq C$, from $(\ref{Yinx-cauchy})$ and $(\ref{estimatemarkovYZK})$, we have that
	\be\label{Zinx-cauchy}
	\lim_{{n,p} \rightarrow +\infty} \E \biggl[\int_0^{+\infty}e^{-r t}|Z^{i,x,n}_t-Z^{i,x,p}_t|^2 dt\biggr]=0.
	\ee
	Thus there exists $Z^{i,x}$ in $\M_r^2$ such that
	\be
	\lim_{{n} \rightarrow +\infty} \E \biggl[ \int_0^{+\infty}e^{-r t}|Z^{i,x,n}_t-Z^{i,x}_t|^2 dt\biggr] =0.
	\ee
	Now we know from $\eqref{systemappro}$ that
	\beqa
	K^{i,x,n}_t&=&\int_0^t e^{-r s}f_i(X^x_s,
	Y^{1,x,n-1}_s,\dots,Y^{i-1,x,n-1}_s,Y^{i,x,n}_s,Y^{i+1,x,n-1}_s,\dots,Y^{m,x,n-1}_s,Z^{i,x,n}_s)ds\nn\\
	&&+e^{-r t}Y^{i,x,n}_t-Y^{i,x,n}_0-\int_0^t e^{-r s}Z^{i,x,n}_sdB_s,
	\eeqa
	we deduce in view of $\eqref{Yinx-cauchy}$ and $\eqref{Zinx-cauchy}$ that
	\be\label{Kinx-cauchy}
	\lim_{n,p \rightarrow +\infty}\E \biggl[\sup_{t\geq0}|K_t^{i,x,n}-K_t^{i,x,p}|^2\biggr]=0.
	\ee
	Then there exists $K^{i,x}\in \K_r^2$ since it is clearly increasing and continuous, such that
	\be\label{uniform_kinx}
	\lim_{n \rightarrow +\infty}\E \biggl[\sup_{t\geq0}|K_t^{i,x,n}-K_t^{i,x}|^2\biggr]=0.
	\ee
	Next taking into account that
	\beqa\nn
	&&\lim_{n\rightarrow+\infty}\E\biggl[\int_{0}^{+\infty}e^{-r s}\Big|f_i(X^x_s,
	Y^{1,x,n-1}_s,\dots,Y^{i-1,x,n-1}_s,Y^{i,x,n}_s,Y^{i+1,x,n-1}_s,\dots,Y^{m,x,n-1}_s,Z^{i,x,n}_s)\\
	&&\qquad \qquad \qquad-f_i(X^x_s,
	Y^{1,x}_s,\dots,Y^{m,x}_s,Z^{i,x}_s)\Big|^2ds\biggr]=0,
	\eeqa
	we can deduce that the limit $Y^{i,x}$ has the following form: for any $i\in \I$
	\be
	e^{-r s}Y^{i,x}_s=\int_s^{+\infty} e^{-r s}f_i(X^x_s,
	Y^{1,x}_s,\dots,Y^{m,x}_s,Z^{i,x}_s)ds+K^{i,x}_\infty-K^{i,x}_s-\int_s^{+\infty} e^{-r s}Z^{i,x}_sdB_s.
	\ee
	Next in view of RBSDE $\eqref{systemappro}$, we have $\forall i\in \I$
	\be
	Y^{i,x,n}_s\geq \underset{j\in\I^{-i}}{\max}\{Y^{j,x,n-1}_s-g_{ij}(X^x_s)\},
	\ee
	thus passing to the limit when $n$ goes to $\infty$, yields
	\be
	Y^{i,x}_s\geq \underset{j\in\I^{-i}}{\max}\{Y^{j,x}_s-g_{ij}(X^x_s)\}.
	\ee
	%Next since $K^{i,x}_t$ is increasing and $(Y^{i,x,n},K^{i,x,n})$ tends to $(Y^{i,x},K^{i,x})$ uniformly in $t$ in probability, then the measure $dK^{i,x,n}$ tends to $dK^{i,x}$ weakly in probability, which implies that
	%\be\nn\ba
	%\int_0^{+\infty}e^{-r s}\Big[Y_s^{i,x,n}-&\underset{j\in\I^{-i}}{\max}\{Y^{j,x,n-1}_s-g_{ij}(X^x_s)\}\Big]^+\, dK^{i,x,n}_s \\ &\longrightarrow \int_0^{+\infty}e^{-r s}\Big[Y_s^{i,x}-\underset{j\in\I^{-i}}{\max}\{Y^{j,x}_s-g_{ij}(X^x_s)\}\Big]^+\, dK^{i,x}_s,
	%\ea\ee
	%in probability as $n \rightarrow +\infty$.\\
	%In addition, having in mind that
	%\be\nn
	%\int_0^{+\infty}e^{-r s}\Big[Y_s^{i,x,n}-\underset{j\in\I^{-i}}{\max}\{Y^{j,x,n-1}_s-g_{ij}(X^x_s)\}\Big]\, dK^{i,x,n}_s=0, \qquad \forall i\in\I,
	%\ee
	%then taking the limit as $n \rightarrow +\infty$, we conclude that
	Next, the uniform convergences $\eqref{uniform_yinx}$ and $\eqref{uniform_kinx}$ imply that, in view of Helly's Convergence
	Theorem (see $\cite{KF}$, pp. $370$),
	\be\nn
	\int_0^{+\infty}e^{-r s}\{Y_s^{i,x}-\underset{j\in\I^{-i}}{\max}(Y^{j,x}_s-g_{ij}(X^x_s))\}\, dK^{i,x}_s=0, \qquad \forall  i\in \I.
	\ee
	Finally, we deduce that the limit $(Y^{i,x},Z^{i,x},K^{i,x})$ exists and is unique, moreover it solves the RBSDE of the form of (\ref{RBSDE-Markov-framework}), which is the desired result.\\
	
	We now show that $(v^i)_{i\in \I}$ are deterministic l.s.c. functions  of polynomial growth, such that:
	\be
	\forall i\in\I,\; \forall x \in\R^k, \quad Y^{i,x}_0=v^i(x).
	\ee
	First we know that by Theorem $\ref{theorem_appendix}$ (see Appendix A ) that there exists deterministic continuous function $v^{i,n}$ of polynomial growth, such that:
	\be\label{YinVin}
     \forall i\in\I,\; \forall x \in\R^k, \qquad Y^{i,x,n}_0=v^{i,n}(x),
	\ee
	In addition, we also know that by Theorem $5.2$ in \cite{P99}, that there exist two deterministic continuous with polynomial growth functions $\bar{v}$ and $\underline{v}$, such that: for any $x\in\R^k$
	\be
	\bar{Y}^{x}_0=\bar{v}(x), \quad \mbox{and} \quad \underline{Y}^{x}_{0}=\underline{v}(x),
	\ee
	where $\bar{Y}^{x}_s$ and $\underbar{Y}^{x}$ are respectively solutions to BSDEs (\ref{solmax}) and (\ref{solmin}).
	
	Next we have that the sequence $(Y^{i,x,n})_{i\in \I}$ is increasing and satisfies $Y^{i,x,n}\le Y^{i,x,n+1}\le \bar{Y}^{x}$. Moreover we obtain by using a comparison result Theorem in \cite{HLW}, that $\underline{Y}^x\le Y^{i,x,n}$ since $\underline{Y}^x$ is the solution of the BSDE associated with $(\underset{i\in\mathcal{I}}{\min}\;f_i,g_{ij})$. Therefore the sequence $(v^{i,n})_{n\geq1}$ is non decreasing and satisfies that
	\be\label{visco-increasing}
	\underline{v}\le v^{i,n}\le\bar{v},
	\ee
	which implies that $(v^{i,n})_{n\geq1}$ converges pointwisely to $v^i$. Consequently $v^i$ is lower semi-continuous on $\R^k$, moreover it is of polynomial growth thanks to (\ref{visco-increasing}).
	Finally, we deduce that
	\be
	\forall i\in\I,\; \forall x \in\R^k, \qquad Y^{i,x}_0=v^{i}(x).
	\ee
	
\end{proof}
%%%%%%%%%%%%%%%%%%%%%%%%%%%%%%%%%%%%%%%%%%%%%%%%%%%%%%%%%%%%%
\section{Uniqueness of the viscosity solution of the PDE $\eqref{PDE}$}
In this section we will show the uniqueness of the viscosity solution of $\eqref{PDE}.$  We first need the following lemma.
\begin{lemma}
	Let $(v^i(x))_{i\in\I}$ be a supersolution of the system $\eqref{PDE}$, then, there exists
	$\lambda_0>0$ which does not depend on $\theta$ such that for any $\lambda \geq \frac{\lambda_0}{r}$ and $\theta>0$, the $m$-uplet $(v^i(x) + \lambda\theta(1+|x|^2))_{i\in\I}$ is a supersolution for $\eqref{PDE}$.
\end{lemma}
\begin{proof}
	Without loss of generality we assume that the functions $v^1,\dots,v^m$ are l.s.c..  Let $i\in\I$ be fixed and let $\varphi\in\mathcal{C}^{1,2}$  be such that the function $\varphi-(v^i + \lambda\theta(1+|x|^2))$ has a local maximum in $x$ which is equal to $0$. As $(v^i)_{i\in\I}$ is a supersolution for $\eqref{PDE}$, then we have: $\forall i\in\I$
	\be\label{PDEphi}
	\ba
	\min\biggr\{&v^i(x)- \max\limits_{j\in{\I}^{-i}}(-g_{ij}(x)+v^j(x)),r(\varphi(x)-\lambda\theta(1+|x|^2))\\ &-\frac{1}{2}Tr\Big[\sigma.\sigma^\top(x)D^{2}_{xx}(\varphi(x)-\lambda\theta(1+|x|^2))\Big]-b^\top(x).D_x(\varphi(x)-\lambda\theta(1+|x|^2))\\ &-f_i\big(x,v^1(x),...,v^m(x),\sigma^\top(x).D_x(\varphi(x)-\lambda\theta(1+|x|^2))\big)\biggl\}=0.
	\ea
	\ee
	which implies that
	\be\label{varphipos1}\ba
	(v^i(x) +\lambda\theta(1+&|x|^2))- \max\limits_{j\in{\I}^{-i}}\big(-g_{ij}(x)+(v^j(x) + \lambda\theta(1+|x|^2))\big)\\ & = v^i(x)- \max\limits_{j\in{\I}^{-i}}(-g_{ij}(x)+v^j(x))\geq0.
	\ea\ee
	On the other hand:
	\be\nn
	\ba
	r(\varphi(x)-&\lambda\theta(1+|x|^2)) -\frac{1}{2}Tr\Big[\sigma.\sigma^\top(x)D^{2}_{xx}(\varphi(x)-\lambda\theta(1+|x|^2))\Big] -b^\top(x).D_x(\varphi(x)-\lambda\theta(1+|x|^2))\\ &-f_i\big(x,v^1(x),...,v^m(x),\sigma^\top(x).D_x(\varphi(x)-\lambda\theta(1+|x|^2))\big)\geq0.
	\ea
	\ee
	Therefore
	\be\label{varphipos2}
	\ba
	r&\varphi(x) -\frac{1}{2}Tr\Big[\sigma.\sigma^\top(x)D^{2}_{xx}\varphi(x)\Big] -b^\top(x).D_x\varphi(x)\\&\qquad\qquad-f_i\big(x,(v^i(x)+\lambda\theta(1+|x|^2))_{i\in\I},\sigma^\top(x).D_x\varphi(x)\big)\\& \geq r\lambda\theta(1+|x|^2)-\frac{1}{2}\lambda\theta Tr\Big[\sigma.\sigma^\top(x)D^{2}_{xx}(1+|x|^2)\Big]-\lambda\theta b^\top(x).D_x(1+|x|^2)\\&\qquad+f_i\big(x,v^1(x),...,v^m(x),\sigma^\top(x).D_x(\varphi(x)-\lambda\theta(1+|x|^2))\big)\\&\qquad\qquad-f_i\big(x,(v^i(x)+\lambda\theta(1+|x|^2))_{i\in\I},\sigma^\top(x).D_x\varphi(x)\big).
	\ea
	\ee
	But,
	\be\nn
	\ba
	&f_i\big(x,v^1(x),...,v^m(x),\sigma^\top(x).D_x(\varphi(x)-\lambda\theta(1+|x|^2))\big)\\ &\qquad\qquad-f_i\big(x,(v^i(x)+\lambda\theta(1+|x|^2))_{i\in\I},\sigma^\top(x).D_x\varphi(x)\big) \\ & =m\lambda\theta C_1(1+|x|^2) - \lambda\theta C_2\sigma^\top(x).D_x(1+|x|^2).
	\ea
	\ee
	where $C_1$ and $C_2$ are bounded by the Lipschitz constant of $f_i$ w.r.t. $\vec{y}$ and $z^i$, which is independent of $\theta$.\\
	Therefore, taking into account the growth conditions on $b$ and $\sigma$, there exists a constant $\lambda_0>0$ which does not depend on $\theta$ such that if $\lambda \geq \frac{\lambda_0}{r}$, the right-hand side of $\eqref{varphipos2}$ is non-negative.  Henceforth, noting that $i$ is arbitrary in $\I$ together with $\eqref{varphipos1}$, we obtain that $(v^i(x) + \lambda\theta(1+|x|^2))_{i\in\I}$  is a viscosity supersolution for $\eqref{PDE}$.
\end{proof}
We now establish the comparison principle between supersolutions and subsolutions of $\eqref{PDE}.$
\begin{proposition}\label{compresult}
	Let $(u^1,\dots,u^m)$ and (resp. $(v^1,\dots,v^m)$) be  a family of u.s.c. viscosity subsolutions (resp. l.s.c. viscosity supersolutions) to $\eqref{PDE}$, and satisfying a polynomial growth condition. Then $\forall i\in\I,$ $u^i\leq v^i$.
\end{proposition}
\begin{proof} Let $\gamma>0$ such that for any $i\in\I$ we have:
	\be\nn
	\forall x\in\R^k,\quad|u^i(x)|+|v^i(x)|\leq C(1+|x|^{\gamma}).
	\ee
	Next we will divide the proof into two steps.\\
	\bf{Step 1:}  To begin with we additionally assume that the functions $f_i$, $i\in\I$ satisfy:\\ $\forall i\in\I$, $\forall x,y^1,\dots,y^{i-1},y^{i+1},\dots,y^m,y,y',z$ if $y\geq y'$ then
	\be\label{mono}
	f_i(x,y^1,\dots,y^{i-1},y,y^{i+1},\dots,y^m,z)-f_i(x,y^1,\dots,y^{i-1},y',y^{i+1},\dots,y^m,z)\leq \lambda(y-y'),
	\ee
	where $\lambda$ is a constant small enough ($\lambda<-mC+r,$ $C$ being the Lipschitz constant of $f_i$).\\
	
	In order to prove the comparison principle, it suffices to show that  for any $i\in\I$, we have:
	$$\forall x\in\R^k,\quad u^i(x)-(v^i(x)+\lambda\theta(1+|x|^2)\leq0,$$ since in taking the limit as $\theta\to0$ we obtain the desired result.  We argue by contradiction and suppose that there exists a point
	$\bar{x}\in\R^k$ such that for $i\in\I$:
	$$\max_{i\in\I}\big(u^i(\bar{x})-w^i(\bar{x})\big)>0,$$ where $w^i:=v^i(x)+\lambda\theta(1+|x|^2).$\\
	We can also assume, without loss of generality, that for every $i\in\I$
	$$\lim_{|x|\to\infty}\big(u^i(x)-w^i(x)\big)=-\infty.$$
	Then, there exists $R>0$ such that
	\be\nn\ba
	\max_{i\in\I}\max_{x\in\R^k}\big(u^i(x)-w^i(x)\big)&=\max_{x\in B(0,R)}\big(u^{j}(x)-w^{j}(x)\big),\\&=u^{j}(\bar{x})-w^{j}(\bar{x}),
	\ea\ee
	where $\bar{x}\in B(0,R)$ and $B(0,R)$  is the open ball in $\R^k$ centered in $0$ and of radius $R$.\\
	Now let us define $\tilde{\I}$ as:
	\be\nn
	\tilde{\I}:=\big\{j\in\I; u^{j}(\bar{x})-w^{j}(\bar{x})=\max_{k\in\I}(u^{k}\big(\bar{x})-w^{k}(\bar{x})\big)\}.
	\ee
	As in $\cite{IK}$, let us show by contradiction that for some $k\in\tilde{\I}$ we have:
	\be\label{barrierebar}
	u^{k}(\bar{x})>\max_{j\in\I^{-k}}\big(u^{j}(\bar{x})-g_{kj}(\bar{x})\big).
	\ee
	Suppose that for any $k\in\tilde{\I}$ we have:
	$$u^{k}(\bar{x})\leq\max_{j\in\I^{-k}}\big(u^{k}(\bar{x})-g_{kj}(\bar{x})\big).$$
	then there exists $l\in\I^{-k}$ such that
	$$u^{k}(\bar{x})-u^{l}(\bar{x})\leq-g_{kl}(\bar{x}).$$
	But $w^{k}$ is a supersolution of $\eqref{PDE}$, therefore we have
	$$w^{k}(\bar{x})-w^{l}(\bar{x})\geq-g_{kl}(\bar{x}).$$
	Then
	$$u^{k}(\bar{x})-u^{l}(\bar{x})\leq-g_{kl}(\bar{x})\leq w^{k}(\bar{x})-w^{l}(\bar{x}).$$
	It implies that
	$$u^{k}(\bar{x})-w^{k}(\bar{x}) \leq u^{l}(\bar{x})-w^{l}(\bar{x}),$$
	which by definition of $\tilde{\I}$ yields that the previous inequality is an equality. Therefore, $l$ also belongs to $\tilde{\I}$ and
	$$u^{k}(\bar{x})-u^{l}(\bar{x})=-g_{kl}(\bar{x}).$$
	Repeating this procedure as many times as necessary and since $\tilde{\I}$ is finite we construct a loop of indices $i_1,\dots,i_p,i_{p+1}$ such that $i_{p+1}=i_1$ $(i_1\neq i_2)$ and
	$$\forall q\in\{1,\dots,p\},\quad u^{i_q}(\bar{x})-u^{i_{q+1}}(\bar{x})=-g_{i_qi_{q+1}}(\bar{x}).$$
	Summing these equalities yields
	$$g_{i_1i_2}(\bar{x})+\dots+g_{i_pi_{p+1}}(\bar{x})=0,$$
	which contradicts Assumption $\bf{[H3]}$ on $(g_{ij})_{i,j\in\I}$, whence the desired result.
	\par Let us now fix $j\in\tilde{\I}$ that satisfies $\eqref{barrierebar}$. For $\epsilon>0$, let us consider the family of u.s.c. functions
	\be\nn
	\Phi_{\epsilon}^j(x,y)=u^j(x)-w^j(y)-\phi_{\epsilon}(x,y),\quad (x,y)\in\R^k\times\R^k.
	\ee
	where $\phi_{\epsilon}(x,y)=\frac{1}{2\epsilon}|x-y|^{2}$. Now let $(x_{\epsilon},y_{\epsilon})\in \bar{B}(0,R)^2$ be such that
	\be\nn
	\Phi_{\epsilon}^j(x_{\epsilon},y_{\epsilon})=\max_{(x,y)\in\bar{B}(0,R)^2}\Phi_{\epsilon}^j(x,y),
	\ee
	which exists since $\Phi_{\epsilon}^j$ is u.s.c. ($\bar{B}(0,R)$ is the closure of $B(0,R)$). Therefore from the inequality $\Phi_{\epsilon}^j(x_{\epsilon},x_{\epsilon})+\Phi_{\epsilon}^j(y_{\epsilon},y_{\epsilon})\leq 2\Phi_{\epsilon}^j(x_{\epsilon},y_{\epsilon}),$ we deduce
	\be\nn
	\frac{1}{\epsilon}|x_{\epsilon}-y_{\epsilon}|^{2}\leq \big(u^{j}(x_{\epsilon})-u^{j}(y_{\epsilon})\big)+\big(w^{j}(x_{\epsilon})-w^{j}(y_{\epsilon})\big).
	\ee
	Consequently,
	$\frac{1}{\epsilon}|x_{\epsilon}-y_{\epsilon}|^{2}$ is bounded (thanks to the growth conditions on $u^{j}$ and $w^{j}$), and as
	$\epsilon\to0$, $(x_{\epsilon}-y_{\epsilon})_{\epsilon}$ converges
	to 0. By the boundedness of the sequences, one can substract subsequences $(x_{\epsilon_n},y_{\epsilon_n})_{n}$ which converges to $(\tilde{x},\tilde{x})$ when $\epsilon\to0$, and we have,
	\be\label{Phixeps}
	\Phi_{\epsilon}^j(\bar{x},\bar{x})=u^{j}(\bar{x})-w^{j}(\bar{x})\leq \Phi_{\epsilon}^j(x_{\epsilon},y_{\epsilon})\leq u^{j}(x_{\epsilon})-w^{j}(y_{\epsilon}).
	\ee
	Then,
	\be\nn
	u^{j}(\bar{x})-w^{j}(\bar{x})\leq u^{j}(\tilde{x})-w^{j}(\tilde{y}),
	\ee
	since $u^{j}$ is u.s.c. and $w^{j}$ is l.s.c..  As the maximum of $u^{j}-w^{j}$ on $\bar{B}(0,R)$ is reached in $\bar{x}$
	then this last inequality is actually an equality. It implies, from the definition of $\phi_{\epsilon}$ and $\eqref{Phixeps}$, that the
	sequence $(x_{\epsilon},y_{\epsilon})_{\epsilon}$ converges to $(\bar{x},\bar{x})$.  From which we deduce
	\be\nn
	\frac{1}{2\epsilon}|x_{\epsilon}-y_{\epsilon}|^{2}\longrightarrow 0, \mbox{ as } \epsilon\to0.
	\ee
	Therefore from $\eqref{Phixeps}$ and the fact that $u^{j}$ is u.s.c. and $w^{j}$ is l.s.c., we obtain,
	\be\nn
	\underset{\epsilon\to0}{\limsup}\, u^{j}(x_{\epsilon})-\underset{\epsilon\to0}{\liminf}\,w^{j}(y_{\epsilon})\leq u^{j}(\bar{x})-w^{j}(\bar{x})\leq \underset{\epsilon\to0}{\liminf}\, u^{j}(x_{\epsilon})-\underset{\epsilon\to0}{\limsup}\,w^{j}(y_{\epsilon}).
	\ee
	It imply that
	\be\nn
	\underset{\epsilon\to0}{\liminf}\, u^{j}(x_{\epsilon})=\underset{\epsilon\to0}{\limsup}\, u^{j}(x_{\epsilon})=u^{j}(\bar{x}),
	\ee
	and
	\be\nn
	\underset{\epsilon\to0}{\liminf}\,w^{j}(y_{\epsilon})=\underset{\epsilon\to0}{\limsup}\,w^{j}(y_{\epsilon})=w^{j}(\bar{x}).
	\ee
	Therefore, we deduce
	\be\nn
	\big(u^{j}(x_{\epsilon}),w^{j}(y_{\epsilon})\big)\longrightarrow \big(u^{j}(\bar{x}),w^{j}(\bar{y})\big), \mbox{ as } \epsilon\to0.
	\ee
	\par Next, as the functions $(u^{k})_{k\in\I}$ are u.s.c. and $(g_{jk})_{j,k\in\I}$ are continuous, and since the index $j$ satisfies $\eqref{barrierebar}$, then there exists $\rho > 0$ such that for $x\in B(\bar{x},\rho)$ we have $$u^{j}(x)>\max_{k\in\I^{-j}}\big(u^{k}(x)-g_{jk}(x)\big).$$
	But by construction we  have $(x_{\epsilon},u^{j}(x_{\epsilon}))\longrightarrow_{\epsilon}(\bar{x},u^{j}(\bar{x}))$ and once more since $u^j$ is u.s.c. then for $\epsilon$
	small enough we have:
	\be\label{ujukgjk}
	u^{j}(x_{\epsilon})>\max_{k\in\I^{-j}}\big(u^{k}(x_{\epsilon})-g_{jk}(x_{\epsilon})\big).
	\ee
	\par Now, From  the definition of $\phi_{\epsilon}$ we have:
	\be\nn
	\left\{\ba &D_x\phi_{\epsilon}(x,y)=\frac{1}{\epsilon}(x-y),\\ & D_y\phi_{\epsilon}(x,y)=-\frac{1}{\epsilon}(x-y),\\ & D(x,y)=D_{x,y}^{2}\phi_{\epsilon}(x,y)=\frac{1}{\epsilon}
	\begin{pmatrix}
		I&-I \\
		-I&I
	\end{pmatrix}.
	\ea
	\right.\ee
	where $I$ stands for the identity matrix.
	Applying the result by Crandall et al. (see e.g. $\cite{CIL}$, theorem $3.2.$) to the function $\Phi_{\epsilon}^j$ in $(x_{\epsilon},y_{\epsilon})$, we denote $q=\frac{1}{\epsilon}(x_{\epsilon}-y_{\epsilon})$ and for any $\kappa>0$ we can find $X, Y\in\Sm_m$, such that:
	\be\label{derive}
	\left\{\ba &(q,X)\in J^{2,+} u^j(x_{\epsilon}),\\ & (q,Y)\in J^{2,-} w^j(y_{\epsilon}),\\ & -(\frac{1}{\kappa}+||D(x_{\epsilon},y_{\epsilon})||)I\leq
	\begin{pmatrix}
		X&0 \\
		0&-Y
	\end{pmatrix}\leq D(x_{\epsilon},y_{\epsilon})+\kappa D(x_{\epsilon},y_{\epsilon})^2.
	\ea
	\right.\ee
	As $(u^i)_{i\in\I}$ (resp. $(w^i)_{i\in\I}$ ) is a subsolution (resp. supersolution) of $\eqref{PDE}$ and taking into account $\eqref{ujukgjk}$, we obtain:
	\be\nn
	 ru^j(x_{\epsilon})-b^\top(x_{\epsilon}).q-\frac{1}{2}Tr[(\sigma\sigma^\top)(x_{\epsilon}).X]-f_j(x_{\epsilon},(u^i(x_{\epsilon}))_{i\in\I},\sigma^\top(x_{\epsilon}).q)\leq0,
	\ee
	and
	\be\nn
	 rw^j(y_{\epsilon})-b^\top(y_{\epsilon}).q-\frac{1}{2}Tr[(\sigma\sigma^\top)(y_{\epsilon}).Y]-f_j(y_{\epsilon},(w^i(y_{\epsilon}))_{i\in\I},\sigma^\top(y_{\epsilon}).q)\geq0,
	\ee
	which implies that
	\be\label{sum12}
	\ba
	 ru^j(x_{\epsilon})&-rw^j(y_{\epsilon})-(b^\top(x_{\epsilon})-b^\top(y_{\epsilon})).q-\frac{1}{2}Tr[(\sigma\sigma^\top)(x_{\epsilon}).X-(\sigma\sigma^\top)(y_{\epsilon}).Y]\\ &-(f_j(x_{\epsilon},(u^i(x_{\epsilon}))_{i\in\I},\sigma^\top(x_{\epsilon}).q)-f_j(y_{\epsilon},(w^i(y_{\epsilon}))_{i\in\I},\sigma^\top(y_{\epsilon}).q))\leq0.
	\ea\ee
	Choosing now $\kappa=\epsilon$ in $\eqref{derive}$ we have
	\be\label{derive2}
	D(x_{\epsilon},y_{\epsilon})+\kappa D(x_{\epsilon},y_{\epsilon})^2\leq\frac{3}{\epsilon}\begin{pmatrix}
		I&-I \\
		-I&I
	\end{pmatrix}.
	\ee
	But, from $\bf{[H1]}$, $\eqref{derive}$ and $\eqref{derive2}$ we get
	\be\label{trace}\frac{1}{2}Tr[(\sigma\sigma^\top)(x_{\epsilon}).X-(\sigma\sigma^\top)(y_{\epsilon}).Y]\leq \frac{C}{\epsilon}|x_{\epsilon}-y_{\epsilon}|^2,\ee
	and
	\be\label{sigmatb}(b^\top(x_{\epsilon})-b^\top(y_{\epsilon})).q +(\sigma^\top(x_{\epsilon})-\sigma^\top(y_{\epsilon})).q\leq \frac{C}{\epsilon}|x_{\epsilon}-y_{\epsilon}|^2,\ee
	where $C$ is a constant which hereafter may change from line to line.\\
	Next, by plugging into $\eqref{sum12}$ we obtain:,
	\be\\\ba
	 ru^j(x_{\epsilon})-rw^j(y_{\epsilon})&-(f_j(x_{\epsilon},(u^j(x_{\epsilon}))_{i\in\I},\sigma^\top(x_{\epsilon}).q)\\&\qquad-f_j(x_{\epsilon},(w^j(y_{\epsilon}))_{i\in\I},\sigma^\top(x_{\epsilon}).q))\leq S_{\epsilon}.
	\ea\ee
	where, from $\eqref{trace}$, $\eqref{sigmatb}$ and the convergence to zero of $(x_{\epsilon}-y_{\epsilon})_{\epsilon}$, we deduce $\underset{\epsilon\to0}{\limsup}\,S_{\epsilon}\leq0$. Next linearizing $f_j$ , which is Lipschitz w.r.t. $(y^i)_{i\in\I}$, and using hypothesis $\eqref{mono}$ we obtain:
	\be\nn
	(r-\lambda)(u^j(x_{\epsilon})-w^j(y_{\epsilon}))-C\sum_{k\in\I^{-j}}(u^k(x_{\epsilon})-w^k(y_{\epsilon}))\leq S_{\epsilon}.
	\ee
	Thus
	\be\nn
	(r-\lambda)(u^j(x_{\epsilon})-w^j(y_{\epsilon}))\leq C\sum_{k\in\I^{-j}}(u^k(x_{\epsilon})-w^k(y_{\epsilon}))^++S_{\epsilon}.
	\ee
	Next taking the limit superior in both hand sides as $\epsilon$, using that $u^k$ (resp. $w^k$ ) is u.s.c. (resp. l.s.c.) and finally $j\in\tilde{\I}$ to obtain:
	\be\nn\ba
	(r-\lambda)(u^j(\bar{x})-w^j(\bar{y}))&\leq C\sum_{k\in\I^{-j}}(u^k(\bar{x})-w^k(\bar{y}))^+,\\ &\leq (m-1)C(u^j(\bar{x})-w^j(\bar{y})).
	\ea\ee
	which is contradictory since $u^j(\bar{x})-w^j(\bar{y})>0$ and $r-\lambda>mC.$ Thus for any $i\in\I$, $u^i\leq w^i$\\ \\
	\bf{Step 2:} The general case.\\
	For arbitrary $\lambda\in\R$, if $(u^i)_{i\in\I}$ (resp. $(v^i)_{i\in\I}$ ) is a subsolution (resp. supersolution) of $\eqref{PDE}$ then $\tilde{u}^{j}(x)=e^{\lambda t} u^{j}(x)$ and $\tilde{v}^{j}(x)=e^{\lambda t} v^{j}(x)$ is a subsolution (resp. supersolution) of the following system of variational inequalities with oblique reflection: $\forall i \in\I$
	\be\label{PDEtilde}
	\ba
	\min\biggr\{\tilde{v}^i(x)- &\max\limits_{j\in{\I}^{-i}}(-e^{\lambda t}g_{ij}(x)+\tilde{v}^j(x)),r\tilde{v}^i(x)-\cL \tilde{v}^i(x)\\ &-e^{\lambda t}f_i(x,(e^{-\lambda t}\tilde{v}^i(x))_{i\in\I},e^{-\lambda t}\sigma^\top(x).D_x\tilde{v}^i(x))\biggl\}=0,
	\ea
	\ee
	Let $\alpha$ be a positive constant. By choosing $r=\alpha+\lambda$, with $\lambda$ small enough we deduce that the function $F_{i}$ defined as follows
	\be\nn
	F_{i}(x,(v^k)_{k\in\I},z)=-\lambda v^i+e^{\lambda t}f_i(x,e^{-\lambda t}(v^k)_{k\in\I},e^{-\lambda t}z),\qquad \forall i\in\I.
	\ee
	satisfies condition $\eqref{mono}$. Thanks to the result stated in Step $1$, we obtain $\tilde{u}^i\leq \tilde{v}^i$, $i\in\I$ from which it follows that $u^i\leq v^i$, for any $i\in\I$, which is the desired result.
\end{proof}
As a by-product of the previous Proposition $\ref{compresult}$, we classically obtain:
\begin{corollary}\label{corol}
	The system of variational inequalities with interconnected obstacles $\eqref{PDE}$ has at most one solution in the class $\Pi^g$. And if the solution in $\Pi^g$ exists, it is necessarily continuous.
\end{corollary}
%%%%%%%%%%%%%%%%%%%%%%%%%%%%%%%%%%%%%%%%%%%%%%%%%%%%%%%%%%%%%
\section{Existence of the viscosity solution of the PDE $\eqref{PDE}$}
%%%%%%%%%%%%%%%%%%%%%%%%%%%%%%%%%%%%%%%%%%%%%%%%%%%%%%%%%%%%%
\begin{theorem}\label{Theo_Ex_uni}
	Assume that \bf{[H2]} and \bf{[H3]} hold. Then the system of variational inequalities with interconnected obstacles $\eqref{PDE}$ has a unique continuous solution $(v^i)_{i\in I}$ in the class $\Pi^g$.
\end{theorem}
\begin{proof}
	Let $(v^i)_{i\in I}$ be the functions constructed in Theorem $\ref{therem31}$ and associated with the solution of the system of RBSDEs with interconnected obstacles $\eqref{1RBSDE}$ (the solution of this system of RBSDEs exists and is unique). The functions
	$v^i,$ $i\in I$ belong to $\Pi^g$ and are thus locally bounded. So we have to show that they
	are viscosity solutions of system $\eqref{PDE}$. \\
	
	Let us show that $(v^i)_{i\in I}$ is a viscosity supersolution of $\eqref{PDE}$. First recall that by Theorem $\ref{therem31}$, for any $i\in I$ $v^i$ is l.s.c., i.e. $$v^i(x)=v^i_*(x):=\underset{x'\to x}{\underline{\lim}}v^i(x').$$
	By construction and for any $i\in I$ it holds
	\be\label{vintovi}
	v^i = \lim_{n \to+\infty}\nearrow v^{i,n};
	\ee
	where $v^{i,n}$ , for $n\geq1$, is defined in $\eqref{YinVin}$. By Theorem $\ref{theorem_appendix}$ (see Appendix A ), $v^{i,n}$ is a viscosity solution of the following variational inequality:
	\be\label{PDEnn}
	\ba
	\min\biggr\{&v^{i,n}(x)- \max\limits_{j\in{\I}^{-i}}\{-g_{ij}(x)+v^{j,n}(x)\},rv^{i,n}(x)-\cL v^{i,n}(x)\\ &-f_i\big(x,(v^{1,n-1},...,v^{i-1,n-1},v^{i,n},v^{i+1,n-1},...,v^{m,n-1})(x),\sigma^\top(x).D_xv^{i,n}(x)\big)\biggl\}=0,
	\ea
	\ee
	Now let us fix $i\in I$, let $x\in\R^k$ and $(q, X) \in J^{2,-} v^{i}(x).$ By $\eqref{vintovi}$ and Lemma $6.1$ in $\cite{CIL}$, there exist sequences:
	$$
	n_{j} \rightarrow+\infty, \quad x_{j} \rightarrow x, \quad\left(q_{j}, X_{j}\right) \in J^{2,-} v^{i}_{n_{j}}\left(x_{j}\right)
	$$
	such that
	$$
	\left(v^{i,n_j}(x_j),q_{j}, X_{j}\right) \rightarrow(v^i(x),q, X).
	$$
	Next from the viscosity supersolution property for $v^{i,n}$ we have:
	\be\nn
	\ba
	 &r\bar{v}^{i,n_j}(x_j)-b^\top(x_j).q_j-\frac{1}{2}\operatorname{Tr}[(\sigma\sigma^\top)(x_j).X_j]\\&-f_i\big(x,(v^{1,n_j-1},...,v^{i-1,n_j-1},v^{i,n_j},v^{i+1,n_j-1},...,v^{m,n_j-1})(x_j),\sigma^\top(x_j).q_j\big)\geq0,
	\ea
	\ee
	which implies that:
	\be\label{supersol}
	\ba
	&r\bar{v}^i(x)-b^\top(x_j).q_j-\frac{1}{2}\operatorname{Tr}[(\sigma\sigma^\top)(x).X]\\&\geq \underset{j\to+\infty}{\underline{\lim}}f_i\big(x,(v^{1,n_j-1},...,v^{i-1,n_j-1},v^{i,n_j},v^{i+1,n_j-1},...,v^{m,n_j-1})(x_j),\sigma^\top(x_j).q_j\big),\\ & = \underset{l\to+\infty}{\underline{\lim}}f_i\big(x,(v^{1,n_{j_l}-1},...,v^{i-1,n_{j_l}-1},v^{i,n_{j_l}},v^{i+1,n_{j_l}-1},...,v^{m,n_{j_l}-1})(x_{j_l}),\sigma^\top(x_{j_l}).q_{j_l}\big),
	\ea
	\ee
	for a subsequence $(j_l)_{l\geq0}$. Using the uniform polynomial growth of $v^{i,n},$ $k\in\I$ there exists a subsequence of $(j_l)_{l\geq0}$ which we still denote $(j_l)_{l\geq0}$ such that for any $k\in\I^{-i}$ the sequence $(v^{k,n_{j_l}-1}(x_{j_l}))_{l\geq0}$ is convergent. Next, using that $v^{k,n}$ is continuous, $v^{k,n}\leq v^{k,n+1}$ and due to the l.s.c. property of $v^{k}$ then for any $x\in\R^k$ and $k\in\I$ (see e.g. $\cite{B}$, p. $91$),
	$$v^k(x)=v^k_*(x)=\underset{\underset{n\to +\infty}{x'\to x}}{\underline{\lim}}v^{k,n}(x').$$
	which yields, for any $k\in\I^{-i}$,
	$$\underset{l\to+\infty}{\lim}\;v^{k,n_{j_l}-1}(x_{j_l})\geq v^k(x).$$
	Going back to $\eqref{supersol}$ and since $f_i$ is continuous and for any $k\in\I^{-i}$, the mapping $y^k \mapsto f_i (x,y^1,...,y^m ,z)$ is non decreasing when all the other variables $(x,(y_j)_{j\in\I})$ are fixed from which we deduce
	\be
	rv^i(x)-b^\top(x).q-\frac{1}{2}\operatorname{Tr}[(\sigma\sigma^\top)(x).X]- f_i(x,(v^k(x))_{k\in\I},\sigma^\top(x).q)\geq0.
	\ee
	On the other hand we know that $v^{i}(x)\geq \max\limits_{j\in{\I}^{-i}}\{v^{j}(x)-g_{ij}(x)\}$, then $v^i$ is a viscosity supersolution of $\eqref{PDE}$. Finally as $i$ is arbitrary in $\I$ then the $m$-tuple $(v^1,...,v^m)$ is a viscosity supersolution for the system of variational inequalities $\eqref{PDE}$.\\
	
	We now show that $(v^{i,*}:=\underset{x'\to x}{\overline{\lim}}v^{i}(x'))_{i\in\I}$ is a subsolution of $\eqref{PDE}$. Since $v^{i,n}\nearrow v^i$ and $v^{i,n}$ is continuous then we have (see e.g. $\cite{B}$, p. $91$)
	\be\label{subsol}
	v^{i,*}(x)=\underset{n\to +\infty}{\lim}{\sup}^*v^{i,n}(x)=\underset{\underset{n\to +\infty}{x'\to x}}{\overline{\lim}}v^{i,n}(x').
	\ee
	From the construction of $v^{i,n}$ in $\eqref{YinVin}$, for any $i\in\I$ and $n\geq0$ we deduce that $$v^{i,n}(x)\geq \max\limits_{j\in{\I}^{-i}}\{v^{j,n}(x)-g_{ij}(x)\},$$
	which implies in taking the limit, $\forall i\in\I$ and $x\in\R^k$
	\be\label{condi1}
	v^{i,*}(x)\geq \max\limits_{j\in{\I}^{-i}}\{v^{j,*}(x)-g_{ij}(x)\}.
	\ee
	Let us now fix $i\in\I$ and let $x\in\R^k$ be such that
	\be\label{barierpos}
	v^{i,*}(x)-\max\limits_{j\in{\I}^{-i}}\{v^{j,*}(x)-g_{ij}(x)\}>0.
	\ee
	Let $(q, X) \in \bar{J}^{2,+} v^{i,*}(x).$ By $\eqref{subsol}$ and Lemma $6.1$ in $\cite{CIL}$, there exist sequences:
	$$
	n_{j} \rightarrow+\infty, \quad x_{j} \rightarrow x, \quad\left(q_{j}, X_{j}\right) \in J^{2,+} v^{i,n_{j}}\left(x_{j}\right)
	$$
	such that
	$$
	\left(v^{i,n_j}(x_j),q_{j}, X_{j}\right) \rightarrow(v^{i,*}(x),q, X).
	$$
	Next from the viscosity supersolution property for $v^{i,n}$ (see $\eqref{PDEnn}$) for any $j\geq0$ we have:
	\be\label{pdenj}
	\ba
	\min\biggr\{&v^{i,n_j}(x_j)- \max\limits_{l\in{\I}^{-i}}\{-g_{il}(x_j)+v^{i,n_j-1}(x)\},rv^{i,n_j}(x_j)+b^\top(x_j).q_j+\frac{1}{2}\operatorname{Tr}[(\sigma\sigma^\top)(x_j).X_j]\\ &-f_i\big(x_j,(v^{1,n_j-1},...,v^{i-1,n_j-1},v^{i,n_j},v^{i+1,n_j-1},...,v^{m,n_j-1})(x_j),\sigma^\top(x_j).q_j\big)\biggl\}\leq0.
	\ea
	\ee
	Next the definition of $v^{l,*}$ implies that
	$$v^{l,*}(x)\geq\underset{n\to +\infty}{\lim\,\sup}^*v^{l,n_j}(x_j),$$
	therefore by $\eqref{barierpos}$, there exists $j_0\geq 0$, such that if $j\geq j_0$ we have
	$$v^{i,n_j}(x_j)>\max\limits_{l\in{\I}^{-i}}\{-g_{il}(x_j)+v^{i,n_j-1}(x_j)\}.$$
	Then $\eqref{pdenj}$ implies that, for any $j\geq j_0$
	\be\nn
	\ba
	&rv^{i,n_j}(x_j)-b^\top(x_j).q_j-\frac{1}{2}\operatorname{Tr}[(\sigma\sigma^\top)(x_j).X_j]\\ &-f_i\big(x_j,(v^{1,n_j-1},...,v^{i-1,n_j-1},v^{i,n_j},v^{i+1,n_j-1},...,v^{m,n_j-1})(x_j),\sigma^\top(x_j).q_j\big)\leq0.
	\ea
	\ee
	Which implies that
	\be\nn
	\ba
	&rv^{i,*}(x)-b^\top(x).q-\frac{1}{2}\operatorname{Tr}[(\sigma\sigma^\top)(x).X]\\ &\leq \underset{j\to +\infty}{\overline{\lim}} f_i\big(x_j,(v^{1,n_j-1},...,v^{i-1,n_j-1},v^{i,n_j},v^{i+1,n_j-1},...,v^{m,n_j-1})(x_j),\sigma^\top(x_j).q_j\big).
	\ea
	\ee
	By the same argument as above (working with an appropriate subsequence, see $\eqref{supersol}$) we deduce that
	\be
	rv^{i,*}(x)-b^\top(x).q-\frac{1}{2}\operatorname{Tr}[(\sigma\sigma^\top)(x).X]- f_i(x,(v^{k,*}(x))_{k\in\I},\sigma^\top(x).q)\leq0.
	\ee
	which, in combination with $\eqref{condi1}$ and since $i$ is arbitrary, means that $(v^i)_{i\in\I}$ is a viscosity subsolution for $\eqref{PDE}$. Thus the $m$-tuple $(v^i)_{i\in\I}$ is a solution for $\eqref{PDE}$ and Corollary $\ref{corol}$ implies that it is continuous and unique.
\end{proof}
\section{Conclusion}
In this paper, we studied a  system of $m$ variational inequalities with interconnected obstacles in infinite horizon \eqref{ourPDE}, associated to optimal multi-modes switching problems, for which we proved the existence of a unique continuous solution in viscosity sense which is the main result of the paper. The proof of that latter strongly relies on the link  between the systems of variational inequalities and reflected backward stochastic differential equations (RBSDEs) with oblique reflection, which we characterized through a Feynman-Kac's formula. The main feature of the system of RBSDEs considered in the paper is that its components are interconnected through both the generators and the obstacles.

Let us comment on some extensions and future work. An interesting open problem would be to treat the general case by only assuming that the functions $f_i$, $i= 1,...,m$ are Lipschitz in $(y^1,\dots, y^m, z^i)$. This would be a challenging problem, since getting rid of the monotonicity assumption \bf{[H2]}(iii) will raise some difficulties related on the one hand to the obtention of the comparison of sub-solutions and super-solutions of system \eqref{ourPDE}. On the other hand there will be a lack of regularity of the functions $(v^i)_{i\in\I}$ as pointed out in the introduction.\\
%which are obtained from system of RBSDEs \eqref{RBSDE-Markov-framework} via Feynman-Kac's formula \eqref{Intro-Feynman-Kac}.\\
Another open interesting problem is the study of infinite horizon RBSDEs with interconnected bilateral obstacles in a Markovian framework and related PDEs with bilateral interconnected obstacles, which are connected with zero-sum stochastic switching games.
It should be noted that this paper is the first step toward the study of these problems.

\begin{appendices}
	\section{Auxiliary Results on Viscosity Solutions of Systems of PDEs}\label{1st-extra-result}
	Let $x\in\R^k$ and let $X^x$ be the solution of the standard SDE $\eqref{sde}$.
	
	Now we are going to introduce the notion of a RBSDE without oblique reflection. Let us introduce the $m$ triplets of processes $((\bar{Y}^{i,x},\bar{Z}^{i,x},\bar{K}^{i,x}))_{i\in\I}$ solution of the following RBSDE: for any $i\in\I$,
	\begin{equation}\label{baRBSDEs}
	\begin{cases}
	\bar{Y}^{i,x}\in\S_r^2,\bar{Z}^{i,x}\in\M_r^2 \mbox{ and } \bar{K}^{i,x}\in\K_r^2;\\
	\begin{aligned}
	e^{-r t}\bar{Y}^{i,x}_t = \int_{t}^{+\infty} &e^{-r s}\bar{f}_i(X^{x}_s,\bar{Y}^{i,x}_s,\bar{Z}^{i,x}_s)ds+ \bar{K}^{i,x}_\infty-\bar{K}^{i,x}_t\\ &-  \int_{t}^{+\infty}e^{-r s}\bar{Z}^{i,x}_sdB_s;
	\end{aligned}\\
	\lim\limits_{t \to+\infty} e^{-r t}\bar{Y}^{i,x}_t=0,\\
	\forall t\geq0, \quad e^{-r t}\bar{Y}^{i,x}_t \geq e^{-r t}\phi(X^{x}_t),\\
	\mbox{ and }
	\int_{0}^{+\infty} e^{-r s}(\bar{Y}^{i,x}_s-\phi(X^{x}_s))d\bar{K}^{i,x}_s = 0,
	\end{cases}
	\end{equation}
	where $\bar{f}_i(x,y,z):\R^k\times\R\times\R^d\mapsto \R$ verifies $\bf{[H2]}$-$(i),(ii)$ and $\phi(x):\R^k\mapsto \R$ is continuous, belongs to $\mathcal{C}^{2}(\R^{k})$, of polynomial growth and satisfies \\
	\bf{[HA]}:$(i)$ For any $x\in\R^{k}$, $\cL \phi(x)\geq0$.\\
	$(ii)$ For any $i\in\mathcal{I}$ the function $\bar{h}_i(x):=\underset{t \to+\infty}{\lim}\bar{Y}^{i,x}_t$ belongs to $\Pi^g$ and satisfies $$\bar{h}_i(x)\geq \phi(x_\infty); \mbox{ where } \phi(x_\infty):=\underset{t \to+\infty}{\lim}\phi(x_t).$$
	\begin{theorem}\label{theorem_appendix}
		For any $x\in\R^k$, the RBSDE $\eqref{baRBSDEs}$ admits a unique solution $((\bar{Y}^{i,x},\bar{Z}^{i,x},\bar{K}^{i,x}))_{i\in\I}$. Moreover, there exists deterministic continuous functions $(\bar{v}^i)_{i\in \I}$ of polynomial growth, defined on $\R^k$, $\R$-valued, such that:
		\be
		\forall i\in\I,\; \forall x \in\R^k,\; \bar{Y}^{i,x}_0=\bar{v}^i(x).
		\ee
		which is moreover solution in viscosity sense of the following system of $m$ partial differential equations: For $i\in\I$ and $x\in\R^k$,
		\be\label{PDE_appendix}
		\ba
		\min\biggr\{\bar{v}^i(x)- &\phi(x),r\bar{v}^i(x)-\cL \bar{v}^i(x)\\ &-\bar{f}_i(x,\bar{v}^i(x),\sigma^\top(x).D_x\bar{v}^i(x))\biggl\}=0,
		\ea
		\ee
		
	\end{theorem}
	\begin{proof}
		Existence and uniqueness of the $m$ triplets of processes $(\bar{Y}^{i,x},\bar{Z}^{i,x},\bar{K}^{i,x})_{i\in\I}$ of $\eqref{baRBSDEs}$ follow from
		Theorem $3.2$ in $\cite{HLW}$. \\
		Now we consider the penalized BSDE: for any $i\in\mathcal{I}$ and $n\geq1$
		\begin{equation}\label{pbsde}\ba
		e^{-r t}\bar{Y}^{i,x,n}_t = \int_{t}^{+\infty} &e^{-r s}[\bar{f}_i(X^{x}_s,\bar{Y}^{i,x,n}_s,\bar{Z}^{i,x,n}_s) +n(\bar{Y}^{i,x,n}_s-\phi(X^{x}_s))^-]\, ds \\ & -\int_{t}^{+\infty}e^{-r s}\bar{Z}^{i,x,n}_s \, dB_s, \qquad \forall t\geq0.
		\ea\end{equation}
		From the result of Chen $\cite{chen}$, for any $n$, BSDE $\eqref{pbsde}$ has a
		unique solution $(\bar{Y}^{i,x,n},\bar{Z}^{i,x,n})$ in the space $\S_r^2\times\M_r^2.$\\
		Next let us define
		$$\bar{K}^{i,x,n}_t=n\int_{0}^{t}e^{-r s}(\bar{Y}^{i,x,n}_s-\phi(X^{x}_s))^-\, ds.$$
		We now show that
		\begin{lemma}\label{estimatt2}
		\begin{itemize}
			\item[(i)] There exist a positive constant $C$, such that the following hold true: for any $i\in\I$ and $n\geq1$.
			\begin{equation}\label{Yi-phi-estime}
				\begin{aligned}
					\E&\biggl[\underset{t\geq0}{\sup}\,e^{-rt}[(\bar{Y}^{i,x,n}_s-\phi(X^{x}_s))^-]^2+n^2\int_{0}^{+\infty}e^{-rs}[(\bar{Y}^{i,x,n}_s-\phi(X^{x}_s))^-]^2ds\biggr]\\ &\leq C\E\biggl[\int_{0}^{+\infty}e^{-rs}\big(|\bar{f}_i(X^x_s,0,0)|^2+|\bar{Y}^{i,x,n}_s|^2+|\bar{Z}^{i,x,n}_s|^2\big)\,ds\biggr],
				\end{aligned}
			\end{equation}
			\item[(ii)] There exist a positive constant $C$ independent of $n$ such that,
			\be\label{estim}\E\biggl[\sup_{t\geq0}e^{-rt}|\bar{Y}_t^{i,x,n}|^2+\int_{0}^{+\infty}e^{-rs}|\bar{Z}^{i,x,n}_s|^2 \, ds+(\bar{K}^{i,x,n}_{\infty})^2\biggr]\leq C,\qquad\forall n\geq1 \mbox{ and } i\in\I.\ee
		\end{itemize}
	\end{lemma}
		\begin{proof}  $(i)$ Fix $x\in\R^k$, and for given $i\in\I$ and $n\geq1$, set
			\begin{equation}\label{yinphi}
				\Y^{i,x,n}_t:=\bar{Y}_t^{i,x,n}-\phi(X^{x}_t), \qquad t\in[0,+\infty).
			\end{equation}
			Next, for $n\geq1$ and $i\in\I$, equation $(\ref{pbsde})$ is equivalent to
			\be\label{dbaryin}\ba d\bar{Y}^{i,x,n}_s=-&[\bar{f}_i(X^x_s,\bar{Y}^{i,x,n}_s,\bar{Z}^{i,x,n}_s)+n(\Y^{i,x,n}_s)^--r \bar{Y}^{i,x,n}_s]ds\\ &+\bar{Z}^{i,x,n}_sdB_s,\ea\ee
			Now $\eqref{dbaryin}$ implies also that
			\be\nn
			\ba
			\Y^{i,x,n}_t = \bar{h}_i(X^{x}_\infty)-&\phi(X^{x}_t)+\int_{t}^{+\infty}  [\bar{f}_i(X^{x}_s,\bar{Y}^{i,x,n}_s,\bar{Z}^{i,x,n}_s) +n(\Y^{i,x,n}_s))^--r \bar{Y}^{i,x,n}_s]\, ds \\ & -\int_{t}^{+\infty}\bar{Z}^{i,x,n}_s \, dB_s, \qquad \forall t\geq0.
			\ea
			\ee
			Then taking the differential with respect to the time variable $t$, gives 
			\be\label{dYixn}\ba d\Y^{i,x,n}_t=&-\sum_{j=1}^{k}\frac{\partial\phi }{\partial x_{j}}(X^{x}_t)\sigma_j(X^{x}_t)dB_t -[\bar{f}_i(X^x_t,\bar{Y}^{i,x,n}_t,\bar{Z}^{i,x,n}_t)+n(\Y^{i,x,n}_t)^--r \bar{Y}^{i,x,n}_t]dt\\ &-\cL\phi(X^{x}_t)dt+\bar{Z}^{i,x,n}_tdB_t.\ea\ee %where we set $\sigma_j(X^{x}_t)dB_t:=\sum_l \sigma_{jl}(X^{x}_t)dB_t^l$ to simplify the notation.\\
			Applying It\^o-Tanaka's formula for $e^{-rt}[(\Y^{i,x,n}_t)^-]^2$ we obtain
			\begin{equation}\label{ito-tanaka}
			\ba
			 d(e^{-rt}[(\Y^{i,x,n}_t)^-]^2)=&-re^{-rt}[(\Y^{i,x,n}_t)^-]^2dt-2e^{-rt}(\Y^{i,x,n}_t)^-d\Y^{i,x,n}_t\\&+\frac{1}{2}dL_t^0(\Y^{i,x,n})+1_{\L_{i,x,n}}e^{-rt}(\bar{Z}^{i,x,n}_t)^2dt;
			\ea
			\end{equation}
		where $L^0(\Y^{i,x,n})$ denotes the local-time at zero of the semi-martingale $\Y^{i,x,n}$ and for $i\in\I$, $\L_{i,x,n}:=\{(s,\omega)\in[0,+\infty)\times\Omega, \mbox{ such that } \Y^{i,x,n}_s<0 \}.$ Noticing that
		the integral with respect to the local-time is nonnegative. Therefore for every $t\geq0$
        	\be\label{tanaka}
			\ba
			e^{-rt}&[(\Y^{i,x,n}_t)^-]^2+2n\int_{t}^{+\infty}e^{-rs}[(\Y^{i,x,n}_s)^-]^2ds+\int_{t}^{+\infty}1_{\L_{i,x,n}}e^{-rs}(\bar{Z}^{i,x,n}_s)^2ds \\ &\leq  2\int_{t}^{+\infty}e^{-rs}(\Y^{i,x,n}_t)^-\big[-\bar{f}_i(X^x_s,\bar{Y}^{i,x,n}_s,\bar{Z}^{i,x,n}_s)+r \bar{Y}^{i,x,n}_s-\cL\phi(X^{x}_s)\big]ds \\ & \qquad+r\int_{t}^{+\infty}e^{-rs}[(\Y^{i,x,n}_s)^-]^2ds+2M^{i,x,n}_{t,\infty};
			\ea
			\ee
			where we have defined 
			\be\label{Mixns} 
			M^{i,x,n}_{t,\infty}:=\int_{t}^{+\infty}e^{-rs}(\Y^{i,x,n}_t)^-\big(\bar{Z}^{i,x,n}_s-\sum_{j=1}^{k}\frac{\partial\phi }{\partial x_{j}}(X^{x}_s)\sigma_j(X^{x}_s)\big)dB_s.
			\ee
			Notice in particular that $(M^{i,x,n}_{0,t})_{t\geq0}$ is indeed a martingale. Then, taking expectation and from $\bf{[HA]}-(i)$ we get
			\be\label{Etanaka}
			\ba
			&\E\big[e^{-rt}[(\Y^{i,x,n}_t)^-]^2\big]+2n\E\big[\int_{t}^{+\infty}e^{-rs}[(\Y^{i,x,n}_s)^-]^2ds\big]+\E\big[\int_{t}^{+\infty}1_{\L_{i,x,n}}e^{-rs}|\bar{Z}^{i,x,n}_s|^2ds\big] \\ &\leq  2\E\big[\int_{t}^{+\infty}e^{-rs}(\Y^{i,x,n}_t)^-\big(|\bar{f}_i(X^x_s,\bar{Y}^{i,x,n}_s,\bar{Z}^{i,x,n}_s)|+r |\bar{Y}^{i,x,n}_s|\big)ds\big] +r\E\big[\int_{t}^{+\infty}e^{-rs}[(\Y^{i,x,n}_s)^-]^2ds\big]\\ &\leq  2C\E\big[\int_{t}^{+\infty}e^{-rs}(\Y^{i,x,n}_t)^-\big(|\bar{f}_i(X^x_s,0,0)|+ |\bar{Y}^{i,x,n}_s|+|\bar{Z}^{i,x,n}_s|\big)ds\big] +r\E\big[\int_{t}^{+\infty}e^{-rs}[(\Y^{i,x,n}_s)^-]^2ds\big]\\ &\leq \frac{1}{2}\E\big[\int_{t}^{+\infty}e^{-rs}\big(|\bar{f}_i(X^x_s,0,0)|^2+|\bar{Y}^{i,x,n}_s|^2+|\bar{Z}^{i,x,n}_s|^2\big)\,ds\big]+C\E\big[\int_{t}^{+\infty}e^{-rs}[(\Y^{i,x,n}_t)^-]^2\,ds\big];
			\ea
			\ee
			where $C$ is a constant independent of $n$ and which might hereafter vary from line to line.\\
			Now, applying Gronwall's inequality, it follows that
			\be\nn
			\E\big[e^{-rt}[(\Y^{i,x,n}_t)^-]^2\big]\leq C\E\big[\int_{0}^{+\infty}e^{-rs}\big(|\bar{f}_i(X^x_s,0,0)|^2+|\bar{Y}^{i,x,n}_s|^2+|\bar{Z}^{i,x,n}_s|^2\big)\,ds\big],
			\ee
			and
			\begin{equation*}
			\begin{aligned}
			n\E\big[\int_{0}^{+\infty}e^{-rs}[(\Y^{i,x,n}_s)^-]^2ds\big] &+\E\big[\int_{0}^{+\infty}1_{\L_{i,n}}e^{-rs}|\bar{Z}^{i,x,n}_s|^2ds\big]\\ &\leq
			C\E\big[\int_{0}^{+\infty}e^{-rs}\big(|\bar{f}_i(X^x_s,0,0)|^2+|\bar{Y}^{i,x,n}_s|^2+|\bar{Z}^{i,x,n}_s|^2\big)\,ds\big].
			\end{aligned}
			\end{equation*}
			Going back to $(\ref{tanaka})$ and applying Burkholder-Davis-Gundy's inequality, we obtain
			\begin{equation}\label{supBDG}
			\E\big[\underset{t\geq0}{\sup}\,e^{-rt}[(\Y^{i,x,n}_t)^-]^2\big]\leq\, C\E\big[\int_{0}^{+\infty}e^{-rs}\big(|\bar{f}_i(X^x_s,0,0)|^2+|\bar{Y}^{i,x,n}_s|^2+|\bar{Z}^{i,x,n}_s|^2\big)\,ds\big].
			\end{equation}
			On the other hand, from $\eqref{Etanaka}$, we deduce that,
			\be\nn
			\ba
			2n\E\big[\int_{0}^{+\infty}e^{-rs}[(\bar{Y}^{i,x,n}_s&-\phi(X^{x}_s))^-]^2ds\big]\leq n\E\big[\int_{0}^{+\infty}e^{-rs}[(\Y^{i,x,n}_s)^-]^2\,ds\big]\\ & +\frac{C}{n}\E\big[\int_{0}^{+\infty}e^{-rs}\big(|\bar{f}_i(X^x_s,0,0)|^2+|\bar{Y}^{i,x,n}_s|^2+|\bar{Z}^{i,x,n}_s|^2\big)\,ds\big].
			\ea
			\ee
			For $n$ large enough we finally deduce that
			\begin{equation*}
			\begin{aligned}
			n^2\E\big[\int_{0}^{+\infty}e^{-rs}[(\Y^{i,x,n}_s)^-]^2ds\big]\leq C\E\big[\int_{0}^{+\infty}e^{-rs}\big(|\bar{f}_i(X^x_s,0,0)|^2+|\bar{Y}^{i,x,n}_s|^2+|\bar{Z}^{i,x,n}_s|^2\big)\,ds\big].
			\end{aligned}
			\end{equation*}
		Next, thanks to $(i)$ we are able to prove $(ii)$. Applying It\^o's formula with $e^{-rt}|\bar{Y}_t^{i,x,n}|^2$, we obtain, for any $t\geq0$
		\begin{equation*}\label{ito_exp_bary}
		\begin{aligned}
		\E[e^{-rt}|\bar{Y}_t^{i,x,n}|^2]&+\E\big[\int_{t}^{+\infty}e^{-rs}(r|\bar{Y}_s^{i,x,n}|^2+|\bar{Z}_s^{i,x,n}|^2)ds\big]\\& = 2\E\big[\int_{t}^{+\infty}e^{-rs}\bar{Y}_s^{i,x,n}\big[\bar{f}_i(X^{x}_s,\bar{Y}^{i,x,n}_s,\bar{Z}_s^{i,x,n})+n(\Y^{i,x,n}_s)^-\big]ds \big]\\& \leq 2\E\big[\int_{t}^{+\infty}e^{-rs}|\bar{Y}^{i,x,n}_s|\big[|\bar{f}_i(X^{x}_s,0,0)|+C(|\bar{Y}^{i,x,n}_s|+|\bar{Z}_s^{i,x,n}|)\big]ds\big]\\ &\qquad +  2n\E\big[\int_{t}^{+\infty}e^{-rs}|\bar{Y}^{i,x,n}_s|(\Y^{i,x,n}_s)^-ds\big],\\ & \leq \E\big[\int_{t}^{+\infty}e^{-rs}\big[(2C^2+2C+\epsilon)|\bar{Y}_s^{i,x,n}|^2+|\bar{f}_i(X^x_s,0,0)|^2+\frac{1}{2}|\bar{Z}^{i,x,n}_s|^2\big]ds\big] \\ & \qquad+n^2\epsilon^{-1}\E\big[\int_{t}^{+\infty}e^{-rs}[(\Y^{i,x,n}_s)^-]^2ds\big].
		\end{aligned}
		\end{equation*}
		By the polynomial growth of $\bar{f}_i$ (see $\eqref{f_of_pg}$) and $\bf{[H5]}$, standard calculations yield
		\be\nn \E\biggl[\sup_{t\geq0}e^{-rt}|\bar{Y}_t^{i,x,n}|^2+\int_{0}^{+\infty}e^{-rs}|\bar{Z}_s^{i,x,n}|^2ds\biggr]\leq C.\ee
		And, from $\eqref{Yi-phi-estime}$, we deduce that
		\be\nn
		n^2\E\biggl[\int_{0}^{+\infty}e^{-rs}[(\Y^{i,x,n}_s)^-]^2ds\biggr]\leq C.
		\ee
		The proof of Lemma $\ref{estimatt2}$ is now complete.
      	\end{proof}
		Note that if we define $$\bar{f}^n_i(x,y,z):=\bar{f}_i(x,y,z)+n(y-\phi(x))^-,$$
		$$\bar{f}^n_i(x,y,z)\leq \bar{f}^{n+1}_i(x,y,z),$$
		and by the comparison Theorem (see e.g. $\cite{HLW}$) we have, for any $n\geq1$, $\bar{Y}^{i,x,n}\leq \bar{Y}^{i,x,n+1}$, then $\bar{Y}^{i,x,n}$ admits $\P$-a.s. a limit denoted $\bar{Y}^{i,x}$. Moreover by Fatou's lemma, we have
		\be\nn
		\E[ \sup_{t\geq0}e^{-r t}|\bar{Y}_t^{i,x}|^2]\le C.
		\ee
		It then follows by dominated convergence that
		\be\label{ynconverg}
		\lim_{n \rightarrow +\infty}\E\biggl[\int_0^{+\infty}e^{-r s}|\bar{Y}_s^{i,x,n}-\bar{Y}_s^{i,x}|^2 \, ds\biggr]=0.
		\ee
		Now let us use It\^o's formula with $e^{-r t}|\bar{Y}_t^{i,x,n}-\bar{Y}_t^{i,x,p}|^2$ and $\eqref{dbaryin}$. It yields that
		\begin{equation}\label{ynypbar}
		\begin{aligned}
		&e^{-r t}|\bar{Y}_t^{i,x,n}-\bar{Y}^{i,x,p}_t|^2+\int_{t}^{+\infty}e^{-r s}(r |\bar{Y}_s^{i,x,n}-\bar{Y}^{i,x,p}_s|^2+|\bar{Z}_s^{i,x,n}-\bar{Z}^{i,x,p}_s|^2)ds\\ &= 2\int_t^{+\infty}e^{-r s}(\bar{Y}_s^{i,x,n}-\bar{Y}^{i,x,p}_s) (\bar{f}_i(X^x_s,\bar{Y}^{i,x,n}_s,\bar{Z}^{i,x,n}_s)
		-\bar{f}_i(X^x_s,\bar{Y}^{i,x,p}_s,\bar{Z}^{i,x,p}_s))\, ds \\ &\quad +2\int_{t}^{+\infty}(\bar{Y}_s^{i,x,n}-\bar{Y}^{i,x,p}_s)(d\bar{K}_s^{i,x,n}-d\bar{K}_s^{i,x,p})\\&\quad-2\int_{t}^{+\infty}e^{-r s}(\bar{Y}_s^{i,x,n}-\bar{Y}^{i,x,p}_s)(\bar{Z}^{i,x,n}_s-\bar{Z}^{i,x,p}_s)\, dB_s,
		\end{aligned}
		\end{equation}
		Taking expectation in both sides of the last equality yields
		\be\label{itoznpbar}\ba
		\E[& e^{-r t}|\bar{Y}_t^{i,x,n}-\bar{Y}^{i,x,p}_t|^2]+\E\biggl[\int_{t}^{+\infty}e^{-r s}(r |\bar{Y}_s^{i,x,n}-\bar{Y}^{i,x,p}_s|^2+|\bar{Z}_s^{i,x,n}-\bar{Z}^{i,x,p}_s|^2)ds\biggr]\\ & \leq 2\E\biggl[\int_t^{+\infty}e^{-r s}|\bar{Y}_s^{i,x,n}-\bar{Y}^{i,x,p}_s ||\bar{f}_i(X^x_s,\bar{Y}^{i,x,n}_s,\bar{Z}^{i,x,n}_s)
		-\bar{f}_i(X^x_s,\bar{Y}^{i,x,p}_s,\bar{Z}^{i,x,p}_s)|\, ds\biggr]\\ & \quad +2\E\biggl[\int_{t}^{+\infty}(\Y^{i,x,n}_s)^-d\bar{K}_s^{i,x,p}\biggr]+2\E\biggl[\int_{t}^{+\infty}(\Y^{i,x,p}_s)^-d\bar{K}_s^{i,x,n}\biggr].
		\ea\ee
		Noting that
		\be\nn\ba
		 2\E\biggl[\int_{t}^{+\infty}&(\Y^{i,x,n}_s)^-d\bar{K}_s^{i,x,p}\biggr]+2\E\biggl[\int_{t}^{+\infty}(\Y^{i,x,p}_s)^-d\bar{K}_s^{i,x,n}\biggr]\\ &\leq 2\E\biggl[ \sup_{t\geq0}(\Y^{i,x,n}_s)^-(\bar{K}_{\infty}^{i,x,p}-\bar{K}_t^{i,x,p})\biggr]+2\E\biggl[ \sup_{t\geq0}(\Y^{i,x,p}_s)^-(\bar{K}_{\infty}^{i,x,n}-\bar{K}_t^{i,x,n})\biggr], \\ & \leq \E\biggl[\big( \sup_{t\geq0}(\Y^{i,x,n}_s)^-\big)^2\biggr]+\E\big[ (\bar{K}_{\infty}^{i,x,p}-\bar{K}_t^{i,x,p})^2\big]\\ & \quad+\E\biggl[\big( \sup_{t\geq0}(\Y^{i,x,p}_s)^-\big)^2\biggr]+\E\big[ (\bar{K}_{\infty}^{i,x,n}-\bar{K}_t^{i,x,n})^2\big].
		\ea\ee
		And from $\eqref{estim}$, we have
		\be\nn
		\E\left[\int_{0}^{+\infty}e^{-r s}[(\Y^{i,x,n}_s)^-]^2\, ds\right]\leq \frac{C}{n^2}.
		\ee
		Thus, using Fatou's lemma we deduce that
		\be\nn
		\E\left[\int_{0}^{+\infty}e^{-r s}(\bar{Y}^{i,x}_s-\phi(X^{x}_s))^-\,ds\right]\leq \underset{n\to +\infty}{\lim\inf}\,\E\left[\int_{0}^{+\infty}e^{-r s}(\Y^{i,x,n}_s)^-\,ds\right]=0.
		\ee
		This implies particularly that
		\be\nn
		\E\left[\int_{0}^{+\infty}e^{-r s}(\bar{Y}^{i,x}_s-\phi(X^{x}_s))^-\,ds\right]=0.
		\ee
		Note that since $\bar{Y}^{i,x}$ (resp. $\phi(x)$) is a c\`adl\`ag process (resp. continuous) we deduce that $e^{-r s}(\bar{Y}^{i,x}_t-\phi(X^{x}_t))^-=0$ for $t\geq0$ and thus $e^{-r t}\bar{Y}^{i,x}_t \geq e^{-r t}\phi(X^{x}_t),\;\forall t\geq0$. Using now
		Dini's theorem and the Lebesgue dominated convergence one to obtain:
		\be\label{ynvarphi}
		\underset{n\to +\infty}{\lim}\mathbb{E}\left[\big(\sup_{t\geq0}(\Y^{i,x,n}_s)^-\big)^2\right]=0
		\ee \\
		Setting, in $\eqref{itoznpbar} $, $t=0$ and choosing $r \geq C^2+C+1$, from $(\ref{estim})$, $(\ref{ynconverg}) $ and $\eqref{ynvarphi}$, we get
		\begin{equation}\label{barZnp}
		\underset{n,\, p \to +\infty}{\lim}\mathbb{E}\left[\int_{0}^{+\infty}e^{-r s}|\bar{Z}_s^{i,x,n}-\bar{Z}_s^{i,x,p}|^2 \, \mathrm{d}{s}\right]=0.
		\end{equation}
		Consequently, the sequence $(\bar{Z}^{i,x,n})_{n\geq1}$ converges in $\M_r^{2}$ to a process which we denote $\bar{Z}^{i,x}$.\\
		Now, going back to $\eqref{ynypbar}$, we deduce that
		\begin{equation}\label{BDG}
		\begin{aligned}
		\E\biggl[\underset{t\geq0}{\sup}&\,e^{-r t}|\bar{Y}_t^{i,x,n}-\bar{Y}_t^{i,x,p}|^2\biggr]\\ &\leq 2\E\biggl[\int_{t}^{+\infty}e^{-r s}|\bar{Y}_s^{i,x,n}-\bar{Y}_s^{i,x,p}||\bar{f}_i(X_s^x,\bar{Y}_s^{i,x,n},\bar{Z}_s^{i,x,n})-\bar{f}_i(X_s^x,\bar{Y}^{i,x,p}_s,\bar{Z}_s^{i,x,p})|ds\biggr]\\ &\qquad +2\int_{t}^{+\infty}(\bar{Y}_s^{i,x,n}-\bar{Y}^{i,x,p}_s)(d\bar{K}_s^{i,x,n}-d\bar{K}_s^{i,x,p})\\ & \qquad+2\E\biggl[\underset{t\geq0}{\sup}\,\bigg|\int_{t}^{+\infty}e^{-r s}(\bar{Y}_s^{i,x,n}-\bar{Y}_s^{i,x,p})(\bar{Z}_s^{i,x,n}-\bar{Z}_s^{i,x,p})dB_s\bigg|\biggr].
		\end{aligned}
		\end{equation}
		Then, by applying the Burkholder-Davis-Gundy's inequality, we get
		\begin{equation}\label{barYnp}
		\underset{n,\, p \to +\infty}{\lim}\E\biggl[\underset{t\geq0}{\sup}\,e^{-r t}|\bar{Y}_t^{i,x,n}-\bar{Y}_t^{i,x,p}|^2\biggr]=0,
		\end{equation}
		which means that $(\bar{Y}^{i,x,n})_{n\geq1}$ is a Cauchy sequence in $\S_r^{2}$. Consequently $\bar{Y}_t^{i,x}$ is continuous.\\
		From the definition of $\bar{K}^{i,x,n}$ and the penalized BSDE $\eqref{pbsde}$, we have
		\begin{equation}\label{barKnp}
		\underset{n,\, p \to +\infty}{\lim}\E\biggl[\underset{t\geq0}{\sup}\,|\bar{K}_t^{i,x,n}-\bar{K}_t^{i,x,p}|^2\biggr]=0.
		\end{equation}
		Then, $(\bar{K}^{i,x,n})_{n\geq1}$ converges to $\bar{K}^{i,x}$ in $\S_r^{2}$.
		So $K^{i}$ is an increasing process, moreover it is continuous, then $K^{i}\in\K_r^{2}$.\\
		Finally, the uniform convergences $\eqref{barYnp}$ and $\eqref{barKnp}$ imply that, in view of Helly's Convergence
		Theorem (see $\cite{KF}$, pp. $370$),
		\be\nn
		\int_0^{+\infty}e^{-r s}(\bar{Y}_s^{i,x}-\phi(X^x_s)))\, d\bar{K}^{i,x}_s=0, \qquad \forall  i\in \I.
		\ee
		
		On the other hand we know that by Theorem $5.2.$ in $\cite{P99}$ that there exists deterministic continuous function $\bar{v}^{i,n}$ of polynomial growth, such that:
		\be\label{barv_bary}
		\forall i\in\I,\; \forall x \in\R^k,\; \bar{Y}_0^{i,x,n}=\bar{v}^{i,n}(x).
		\ee
		which is moreover solution in viscosity sense of the following parabolic PDE:
		\be
		r\bar{v}^{i,n}(x)-\cL\bar{v}^{i,n}(x)-\bar{f}^n_i(x,\bar{v}^{i,n}(x),\sigma^\top(x).D_x\bar{v}^{i,n}(x))=0.
		\ee
		Since the sequence $(\bar{Y}^{i,x,n})_{i\in \I}$ is increasing and $\eqref{barv_bary}$ imply that $\bar{v}^{i,n}\le \bar{v}^{i,n+1} $.\\
		Next for any $x\in\R^k$, let us set
		\be\nn
		\bar{v}^{i}(x)=\underset{n\to+\infty}{\lim}\bar{v}^{i,n}(x).
		\ee
		Then, it holds that for any $i\in\I$
		\be
		\forall x \in\R^k,\; \bar{Y}_0^{i,x}=\bar{v}^{i}(x).
		\ee
		\begin{lemma}\label{lemmaA}
			For $i\in\I$, the function $\bar{v}^{i}$ is continuous in $\R^k$.
		\end{lemma}
		\begin{proof}
			It suffices to show that whenever, $x_n\to x,\quad\E\big[\underset{t\geq0}{\sup}|\bar{Y}^{i,x_n}_t-\bar{Y}^{i,x}_t|^2\big]\to 0.$\\
			
			Now let us use It\^o's formula with $e^{-\lambda t}|\bar{Y}^{i,x_n}_t-\bar{Y}^{i,x}_t|^2$ and $\eqref{dbaryin}$. It yields, for any $\lambda\geq r$, that
			\begin{equation}\label{yxnyxbar}
			\begin{aligned}
			&e^{-\lambda t}|\bar{Y}_t^{i,x_n}-\bar{Y}^{i,x}_t|^2+\int_{t}^{+\infty}e^{-\lambda s}\{(2r-\lambda) |\bar{Y}_s^{i,x_n}-\bar{Y}^{i,x}_s|^2+|\bar{Z}_s^{i,x_n}-\bar{Z}_s^{i,x}|^2\}ds\\ &= 2\int_t^{+\infty}e^{-\lambda s}(\bar{Y}_s^{i,x_n}-\bar{Y}^{i,x}_s) (\bar{f}_i(X^{x_n}_s,\bar{Y}^{i,x_n}_s,\bar{Z}^{i,x_n}_s)
			-\bar{f}_i(X^x_s,\bar{Y}^{i,x}_s,\bar{Z}_s^{i,x}))\, ds \\ &\quad +2\int_{t}^{+\infty}e^{(r-\lambda) s}(\bar{Y}_s^{i,x_n}-\bar{Y}^{i,x}_s)(d\bar{K}_s^{i,x_n}-d\bar{K}_s^{i,x})\\&\quad-2\int_{t}^{+\infty}e^{-\lambda s}(\bar{Y}_s^{i,x_n}-\bar{Y}^{i,x}_s)(\bar{Z}^{i,x_n}_s-\bar{Z}_s^{i,x})\, dB_s.
			\end{aligned}
			\end{equation}
			Taking expectation in both sides of the last equality yields
			\be\nn\ba
			\E[& e^{-\lambda t}|\bar{Y}_t^{i,x_n}-\bar{Y}^{i,x}_t|^2]+\E\biggl[\int_{t}^{+\infty}e^{-\lambda s}\{(2r-\lambda) |\bar{Y}_s^{i,x_n}-\bar{Y}^{i,x}_s|^2+|\bar{Z}_s^{i,x_n}-\bar{Z}_s^{i,x}|^2\}ds\biggr]\\ & \leq 2\E\biggl[\int_t^{+\infty}e^{-\lambda s}|\bar{Y}_s^{i,x_n}-\bar{Y}^{i,x}_s ||\bar{f}_i(X^{x_n}_s,\bar{Y}^{i,x_n}_s,\bar{Z}^{i,x_n}_s)
			-\bar{f}_i(X^{x}_s,\bar{Y}^{i,x}_s,\bar{Z}^{i,x}_s)|\, ds\biggr]\\ & \quad +2\E\biggl[\int_{t}^{+\infty}e^{(r-\lambda) s}(\phi(X^{x_n}_s-\phi(X^{x}_s))d(\bar{K}_s^{i,x_n}-\bar{K}_s^{i,x})\biggr],
			\\ & \leq \E\biggl[\int_t^{+\infty}e^{-\lambda s}|\bar{f}_i(X^{x_n}_s,\bar{Y}^{i,x}_s,\bar{Z}^{i,x}_s)
			-\bar{f}_i(X^{x}_s,\bar{Y}^{i,x}_s,\bar{Z}^{i,x}_s)|^2\, ds\biggr]\\ &\quad + \E\biggl[\int_t^{+\infty}e^{-\lambda s}\big\{(2C^2+C+1)|\bar{Y}_s^{i,x_n}-\bar{Y}^{i,x}_s |^2+\frac{1}{2}|\bar{Z}^{i,x_n}_s-\bar{Z}^{i,x}_s|^2\, \big\}ds\biggr]\\ & \quad +2\E\biggl[\sup_{t\geq0}e^{(r-\lambda) t}\big|\phi(X^{x_n}_t)-\phi(X^{x}_t)\big|(\bar{K}_{+\infty}^{i,x_n}+\bar{K}_{+\infty}^{i,x})\biggr].
			\ea\ee
			Arguments already used in the proof of $\eqref{estim}$  lead to
			\be\label{ybarxn}\ba
			\E\big[\underset{t\geq0}{\sup}e^{-\lambda t}|\bar{Y}^{i,x_n}_t-&\bar{Y}^{i,x}_t|^2\big] \leq\; \E\biggl[\int_t^{+\infty}e^{-\lambda s}|\bar{f}_i(X^{x_n}_s,\bar{Y}^{i,x}_s,\bar{Z}^{i,x}_s)
			-\bar{f}_i(X^{x}_s,\bar{Y}^{i,x}_s,\bar{Z}^{i,x}_s)|^2\, ds\biggr]\\ & \; +2\E\biggl[\sup_{t\geq0}e^{-p(\lambda-r) t}|\phi(X^{x_n}_t)-\phi(X^{x}_t)|^p\biggr]^{\frac{1}{p}}\E\biggl[(\bar{K}_{+\infty}^{i,x_n}+\bar{K}_{+\infty}^{i,x})^{\frac{p}{p-1}}\biggr]^{\frac{p-1}{p}},
			\ea\ee
			with $p\in ]1,2[$. In the right-hand side of $\eqref{ybarxn}$ the first term converges to 0 as $x_n\to x$ since $\bar{f}_i$ is continuous and of polynomial growth. Next let us show that,
			$$\E\biggl[\sup_{t\geq0}e^{-p(\lambda-r) t}|\phi(X^{x_n}_t)-\phi(X^{x}_t)|^p\biggr]^{\frac{1}{p}}\to 0 \mbox{  as  } x_n\to x.$$
			For any $T\geq0$, there exists $C_p$ (a constant depending only on $p$) such that,
			\be\nn\ba
			\E\biggl[\sup_{t\geq0}e^{-p(\lambda-r) t}|\phi(X^{x_n}_t)-\phi(X^{x}_t)|^p\biggr]&\leq  \; \E\biggl[\sup_{t\leq T}e^{-p(\lambda-r) t}|\phi(X^{x_n}_t)-\phi(X^{x}_t)|^p\biggr]\\ & \quad +\E\biggl[\sup_{t\geq T}e^{-p(\lambda-r) t}|\phi(X^{x_n}_t)-\phi(X^{x}_t)|^p\biggr],\\ & \leq  \; \E\biggl[\sup_{t\leq T}e^{-p(\lambda-r) t}|\phi(X^{x_n}_t)-\phi(X^{x}_t)|^p\biggr] + C_pe^{-\frac{p}{2}(\lambda-r) T};
			\ea\ee
			since $\phi$ belongs to $\S_r^2$.  Now for any $\rho>0$ we have:
			\be\label{varphixn}\ba
			\E\biggl[\sup_{t\leq T}&\;e^{-p(\lambda-r) t}|\phi(X^{x_n}_t)-\phi(X^{x}_t)|^p\biggr]\\ & =  \; \E\biggl[\sup_{t\leq T}e^{-p(\lambda-r) t}|\phi(X^{x_n}_t)-\phi(X^{x}_t)|^p1_{[\underset{t \le T}{\sup}|X^{x_n}_t|+\underset{t \le T}{\sup}|X^{x}_t|\leq \rho]}\biggr] \\ & \quad +\E\biggl[\underset{t \le T}{\sup}e^{-p(\lambda-r) t}|\phi(X^{x_n}_t)-\phi(X^{x}_t)|^p1_{[\underset{t \le T}{\sup}|X^{x_n}_t|+\underset{t \le T}{\sup}|X^{x}_t|> \rho]}\biggr].
			\ea\ee
			But, since $\phi$ is continuous, it is uniformly continuous on compact subsets, then there exists $\pi:\R^k\to\R$ increasing with $\pi(0)=0$, such that:$$|\phi(X^{x_n}_t)-\phi(X^{x}_t)|\leq\pi(|X^{x_n}_t-X^{x}_t|),$$ we have
			\be\nn\ba
			\E\biggl[\sup_{t\leq T}&\,e^{-p(\lambda-r) t}|\phi(X^{x_n}_t)-\phi(X^{x}_t)|^p1_{[\underset{t \le T}{\sup}|X^{x_n}_t|+\underset{t \le T}{\sup}|X^{x}_t|\leq \rho]}\biggr] \\ & \leq \E\biggl[\sup_{t\leq T}\pi(|X^{x_n}_t-X^{x}_t|)1_{[\underset{t \le T}{\sup}|X^{x_n}_t|+\underset{t \le T}{\sup}|X^{x}_t|\leq \rho]}\biggr],
			\\ & \leq \E\biggl[\pi(\sup_{t\leq T}|X^{x_n}_t-X^{x}_t|)1_{[\underset{t \le T}{\sup}|X^{x_n}_t|+\underset{t \le T}{\sup}|X^{x}_t|\leq \rho]}\biggr].
			\ea\ee
			Using the continuity proprety $\eqref{estimat3}$, $\pi(0)=0$ and the Lebesgue dominated convergence theorem to obtain that
			\be
			\E\biggl[\sup_{t\leq T} e^{-p(\lambda-r) t}|\phi(X^{x_n}_t)-\phi(X^{x}_t)|^p1_{[\underset{t \le T}{\sup}|X^{x_n}_t|+\underset{t \le T}{\sup}|X^{x}_t|\leq \rho]}\biggr]\to 0 \mbox{  as  } x_n\to x.
			\ee
			The second term of $\eqref{varphixn}$ satisfies:
			\be\nn\ba
			\E\biggl[\sup_{t\leq T}&\,e^{-p(\lambda-r) t}|\phi(X^{x_n}_t)-\phi(X^{x}_t)|^p1_{[\underset{t \le T}{\sup}|X^{x_n}_t|+\underset{t \le T}{\sup}|X^{x}_t|> \rho]}\biggr] \\ & \leq \E\biggl[\sup_{t\leq T}e^{-2(\lambda-r) t}|\phi(X^{x_n}_t)-\phi(X^{x}_t)|^2\biggr]^{\frac{p}{2}}\E\biggl[1_{[\underset{t \le T}{\sup}|X^{x_n}_t|+\underset{t \le T}{\sup}|X^{x}_t|> \rho]}\biggr]^{\frac{2-p}{2}}, \\ & \leq \rho^{-\frac{2-p}{2}}\E\biggl[\sup_{t\leq T}e^{-2(\lambda-r) t}|\phi(X^{x_n}_t)-\phi(X^{x}_t)|^2\biggr]^{\frac{p}{2}}\E\biggl[\underset{t \le T}{\sup}|X^{x_n}_t|+\underset{t \le T}{\sup}|X^{x}_t|\biggr]^{\frac{2-p}{2}}, \\ & \leq \rho^{-\frac{2-p}{2}}C_x,
			\ea\ee
			when $x_n\to x.$  However, from previous results we have,
			\be\nn
			\underset{x_n\to x}{\limsup}\;\E\biggl[\sup_{t\geq0}e^{-p(\lambda-r) t}|\phi(X^{x_n}_t)-\phi(X^{x}_t)|^p\biggr]\leq  \; \rho^{-\frac{2-p}{2}}C_x + C_pe^{-\frac{p}{2}(\lambda-r) T}.
			\ee
			As $\rho$ and $T$ are arbitrary, then making $\rho\to+\infty$ and $T\to+\infty$ to obtain that,
			\be\label{varphix0}
			\underset{x_n\to x}{\lim}\;\E\biggl[\sup_{t\geq0}e^{-p(\lambda-r) t}|\phi(X^{x_n}_t)-\phi(X^{x}_t)|^p\biggr]=0.
			\ee
			The proof of lemma $\ref{lemmaA}$ is now complete.
		\end{proof}
		\par It remains to show that $(\bar{v}^i)_{i\in \I}$ is exist in viscosity sense. For this, From the previous results we have, for each $x\in\R^k$
		\be\nn\bar{v}^{i,n}(x)\nearrow\bar{v}^{i}(x),\quad\mbox{as}\quad n\to+\infty.\ee
		Since $\bar{v}^{i,n}$ and $\bar{v}^{i}$ are continuous, it follows from Dini's theorem that the above convergence is uniform on compacts.\\
		
		We now show that $\bar{v}^{i}$ is a subsolution of $\eqref{PDE_appendix}$ Let $x$ be a point at which $\bar{v}^{i}(x)>\phi(x)$, and let $(q, X) \in J^{2,+} \bar{v}^{i}(x)$. From Lemma $6.1$ in $\cite{CIL}$, there exist sequences:
		$$
		n_{j} \rightarrow+\infty, \quad x_{j} \rightarrow x, \quad\left(q_{j}, X_{j}\right) \in J^{2,+} \bar{v}^{i}_{n_{j}}\left(x_{j}\right)
		$$
		such that
		$$
		\left(q_{j}, X_{j}\right) \rightarrow(q, X).
		$$
		But for any $j$,
		\be\nn
		 r\bar{v}^{i,n_j}(x_j)-b^\top(x_j).q_j-\frac{1}{2}\operatorname{Tr}[(\sigma\sigma^\top)(x_j).X_j]-\bar{f}^{n_j}_i(x_j,\bar{v}^{i,n_j}(x_j),\sigma^\top(x_j).q_j)\leq0.
		\ee
		Then,
		\be\nn\ba
		 r\bar{v}^{i,n_j}(x_j)-b^\top(x_j).q_j&-\frac{1}{2}\operatorname{Tr}[(\sigma\sigma^\top)(x_j).X_j]\\&-\bar{f}_i(x_j,\bar{v}^{i,n_j}(x_j),\sigma^\top(x_j).q_j)-n_{j}\big(\bar{v}^{i,n_j}(x_j)-\phi(x_{j})\big)^{-} \leq 0.
		\ea\ee
		From the assumption that $\bar{v}^{i}(x)>\phi(x)$ and the uniform convergence of $\bar{v}^{i,n}$, it follows that for $j$ large enough $\bar{v}^{i,n_j}(x_{j})>\phi(x_{j})$. Hence, taking the limit as $j \rightarrow+\infty$ in the above inequality yields:
		\be\nn
		r\bar{v}^{i}(x)-b^\top(x).q-\frac{1}{2}\operatorname{Tr}[(\sigma\sigma^\top)(x).X]-\bar{f}_i(x,\bar{v}^{i}(x),\sigma^\top(x).q)\leq0.
		\ee
		Then $\bar{v}^{i}$ is a subsolution of $\eqref{PDE_appendix}$.
		\par We now show that $\bar{v}^{i}$ is a supersolution of $\eqref{PDE_appendix}$. Let $x$ be arbitrary in $\mathbb{R}^{k}$, and $(q, X) \in J^{2,-} \bar{v}^{i}(x) .$ We already know that $\bar{v}^{i}(x) \geq \phi(x) .$ By the same argument as above, there exist sequences:
		$$
		n_{j} \rightarrow+\infty, \quad x_{j} \rightarrow x, \quad\left(q_{j}, X_{j}\right) \in J^{2,-} \bar{v}^{i}_{n_{j}}\left(x_{j}\right)
		$$
		such that
		$$
		\left(q_{j}, X_{j}\right) \rightarrow(q, X).
		$$
		But for any $j$,
		\be\nn
		 r\bar{v}^{i,n_j}(x_j)-b^\top(x_j).q_j-\frac{1}{2}\operatorname{Tr}[(\sigma\sigma^\top)(x_j).X_j]-\bar{f}^{n_j}_i(x_j,\bar{v}^{i,n_j}(x_j),\sigma^\top(x_j).q_j)\geq0.
		\ee
		Then,
		\be\nn\ba
		 r\bar{v}^{i,n_j}(x_j)-b^\top(x_j).q_j&-\frac{1}{2}\operatorname{Tr}[(\sigma\sigma^\top)(x_j).X_j]\\&-\bar{f}_i(x_j,\bar{v}^{i,n_j}(x_j),\sigma^\top(x_j).q_j)-n_{j}\big(\bar{v}^{i,n_j}(x_j)-\phi(x_{j})\big)^{-} \geq 0.
		\ea\ee
		Hence, taking the limit as $j \rightarrow+\infty$ in the above inequality yields:
		\be\nn
		r\bar{v}^{i}(x)-b^\top(x).q-\frac{1}{2}\operatorname{Tr}[(\sigma\sigma^\top)(x).X]-\bar{f}_i(x,\bar{v}^{i}(x),\sigma^\top(x).q)\geq0.
		\ee
		Then $\bar{v}^{i}$ is a supersolution of $\eqref{PDE_appendix}$.
	\end{proof}
	\section{Comparison Theorem of Infinite Horizon RBSDE}\label{2nd-extra-result}
	In this section we discuss the comparison theorem of the following infinite horizon RBSDE
	\begin{equation}\label{standard_RBSDE}
	\begin{cases}
	Y\in\S_0^2,Z\in\M_0^2 \mbox{ and } K\in\K_0^2;\\
	Y_t = \xi + \int_{t}^{+\infty} f(s,X_s,Y_s,Z_s)ds+ K_\infty-K_t -  \int_{t}^{+\infty}Z_sdB_s;
	\\
	\forall\; t\geq0, \quad Y_t \geq S_t,\\
	\mbox{ and }
	\int_{0}^{+\infty} (Y_s-S_t)dK_s = 0,
	\end{cases}
	\end{equation}
	where $\xi\in\L^2:= \{\xi,\; \xi \mbox{ is a }\F_{\infty}\mbox{-measurable random variable } \xi \mbox{ s.t. } \E[|\xi|^2]<+\infty\}$ and $f$ is a map from $[0,+\infty)\times\R^k\times\R\times\R^d\mapsto \R$  which satisfies to:\\
	\bf{[A1]}: for any $x\in\R^k$ and $t\geq0$, $f(t,x,0,0)$ belongs to $\mathcal{M}_0^2$.\\
	\bf{[A2]}: There exists a positive constant $C$, such that $\forall\left(x,y_{i}, z_{i}\right) \in\R^k\times\R\times\R^d$, $i=1,2$,
	$$
	\left|f(t,x, y_{1}, z_{1})-f(t,x, y_{2}, z_{2})\right| \leq u(t)(\left|y_{1}-y_{2}\right|+\left|z_{1}-z_{2}\right|), \quad \forall t\geq0;$$
	$u(t):[0,+\infty)\mapsto[0,+\infty)$ is a deterministic and bounded function satisfies: $\int_0^{+\infty}u(t)dt<+\infty$ and $\int_0^{+\infty}u^2(t)dt<+\infty$.
	
	Let $\{X_t; t\geq 0\}$ be the solution of the following
	standard SDE:
	\begin{equation}\label{AppB_sde}
	dX_t=b(t,X_t)dt+\sigma(t,X_t)dB_t; \quad X_0=x,\, x\in\R^k, t\geq 0,
	\end{equation}
	where the functions $b$ and $\sigma$ are the ones of
	$\bf{[H1]}$.\\
	We consider a barrier $\{S_t ,t\geq0\},$ which is a continuous progressively measurable real-valued process satisfying\\
	\bf{[A3]}: $\E[\sup\limits_{t\geq0} (S_t^+)^2]<\infty$ and $\underset{t\to +\infty}{\limsup}\; S_t \leq \xi,$ a.s.\\
	
	Then we have the
	\begin{theorem}
		(see e.g. $\cite{HLW}$): Assume $\xi\in\L^2$, \bf{[A1]}, \bf{[A2]} and \bf{[A3]}, then the RBSDE $\eqref{standard_RBSDE}$ associated with $(f,\xi,S)$ has a unique solution $(Y,Z,K).$
	\end{theorem}
	We next prove a comparison theorem
	\begin{theorem}\label{comparison_theo}
		Let $(\xi, f, S)$ and $\left(\xi^{\prime}, f^{\prime}, S^{\prime}\right)$ be two sets of data, each one satisfying all the assumptions $\xi,\xi^{\prime}\in\L^2$, \bf{[A1]}, \bf{[A2]} and \bf{[A3]} (with the exception that the assumption \bf{[A3]} could be satisfied by either $f$ or $f^{\prime}$ only), and suppose in addition the following:
		\begin{itemize}
			\item [$(i)$] $\xi \leq \xi^{\prime}$,
			\item [$(ii)$] $f(t,x, y, z) \leq f^{\prime}(t,x, y, z)$, $\forall(t,x,y, z) \in [0,+\infty)\times\R^k\times\mathbb{R} \times \mathbb{R}^{d}$,
			\item [$(iii)$] $S_{t} \leq S_{t}^{\prime}, \forall t\geq0$.
		\end{itemize}
		Let $(Y, Z, K)$ be a solution of the $R B S D E$ with $\operatorname{data}\; (\xi, f, S)$ and $\left(Y^{\prime}, Z^{\prime}, K^{\prime}\right)$ a solution of the RBSDE with data $\left(\xi^{\prime}, f^{\prime}, S^{\prime}\right)$. Then
		$Y_{t} \leq Y_{t}^{\prime}, \quad \forall t\geq0 \quad$ a.s.
	\end{theorem}
	\begin{proof}
		Let $\bar{Y}_t = Y_t-Y^{\prime}_t$, $\bar{Z}_t = Z_t-Z^{\prime}_t$ and $\bar{K}_t = K_t-K^{\prime}_t$, then $(\bar{Y},\bar{Z},\bar{K})$ satisfies
		\be
		\bar{Y}_t = \xi-\xi^{\prime} + \int_{t}^{+\infty}(\alpha_s\bar{Y}_s+\beta_s\bar{Z}_s+\gamma_s)ds+ (K_\infty-K_t) -(K^{\prime}_\infty-K^{\prime}_t)-  \int_{t}^{+\infty}\bar{Z}_sdB_s;
		\ee
		where
		\be\nn
		\alpha_s = \left\{\ba &\frac{f(s,X_s,Y_s,Z_s)-f(s,X_s,Y^{\prime}_s,Z_s)}{\bar{Y}_s}\quad if\quad \bar{Y}_s\ne0, \\
		& 0 \qquad\qquad\qquad\qquad\qquad\qquad otherwise.
		\ea\right.
		\qquad ,\ee
		\be\nn
		\beta_s = \left\{\ba &\frac{f(X_s,Y^{\prime}_s,Z_s)-f(X_s,Y^{\prime}_s,Z^{\prime}_s)}{\bar{Z}_s}\quad if\quad \bar{Z}_s\ne0, \\
		& 0 \qquad\qquad\qquad\qquad\qquad\qquad otherwise.
		\ea\right.
		\ee
		and $$\gamma_s = f(s,X_s,Y^{\prime}_s,Z^{\prime}_s)-f^{\prime}(s,X_s,Y^{\prime}_s,Z^{\prime}_s).$$
		Applying It\^o-Tanaka's formula to $(\bar{Y}^+_t)^2$, we have:
		\be
		d(\bar{Y}^+_t)^2=1_{\{\bar{Y}_t>0\}}[-2\bar{Y}^+_t(\alpha_t\bar{Y}_t+\beta_t\bar{Z}_t+\gamma_t)dt-2\bar{Y}^+_t d\bar{K}_t +2\bar{Y}^+_t\bar{Z}_t\,dB_t+\bar{Z}^2_t\,dt].
		\ee
		Next, let $$\Gamma_{t,s}:= \exp\big\{\int_{t}^{s}(\alpha_r-\frac{1}{2}\beta^2_r)dr+\int_{t}^{s}\beta_r\,dB_r\big\}, \qquad 0\leq t\leq s, $$
		be the solution of the following SDE
		\be\nn
		\begin{cases}
			d\Gamma_{t,s}=\alpha_s\Gamma_{t,s}\,ds+\beta_s\Gamma_{t,s}\,dB_s;
			\\
			\Gamma_{t,t} = 1 .
		\end{cases}
		\ee
		By Integration by Parts formula, we obtain
		\be\nn
		\ba
		(\bar{Y}^+_t)^2=\;&\Gamma_{t,s}(\bar{Y}^+_s)^2+2\int_{t}^{s}1_{\{\bar{Y}_r>0\}}\Gamma_{t,r}\bar{Y}^+_r(\alpha_r\bar{Y}_r+\beta_r\bar{Z}_r+\gamma_r)dr  +2\int_{t}^{s}1_{\{\bar{Y}_r>0\}}\Gamma_{t,r}\bar{Y}^+_r\,d\bar{K}_r\\&-2\int_{t}^{s}1_{\{\bar{Y}_r>0\}}\Gamma_{t,r}\bar{Y}^+_r\bar{Z}_r\,dB_r-\int_{t}^{s}1_{\{\bar{Y}_r>0\}}\Gamma_{t,r}\bar{Z}^2_r\,dr-\int_{t}^{s}\alpha_r\Gamma_{t,r}(\bar{Y}^+_r)^2\,dr \\& -\int_{t}^{s}\beta_r\Gamma_{t,r}(\bar{Y}^+_r)^2\,dB_r-2\int_{t}^{s}\beta_r\Gamma_{t,r}\bar{Y}^+_r\bar{Z}_r\,dr.
		\ea
		\ee
		Since $f(s,X_s,Y^{\prime}_s,Z^{\prime}_s)\leq f^{\prime}(s,X_s,Y^{\prime}_s,Z^{\prime}_s)$, $\xi \leq \xi^{\prime}$ and,  on  $\{Y_t > Y^{\prime}_t\}$, $Y_t > S^{\prime}_t$, we have
		$$\int_{t}^{s}1_{\{\bar{Y}_r>0\}}\Gamma_{t,r}\bar{Y}^+_r\,d\bar{K}_r\leq0.$$
		Then, taking the expectation for $s\to+\infty$, we have
		\be\nn
		 \E\big[(\bar{Y}^+_t)^2\big]+E\big[\int_{t}^{+\infty}1_{\{\bar{Y}_r>0\}}\Gamma_{t,r}\bar{Z}^2_r\,dr\big]\leq\;\E\big[\int_{t}^{+\infty}1_{\{\bar{Y}_r>0\}}\alpha_r\Gamma_{t,r}(\bar{Y}^+_r)^2\,dr\big].
		\ee
		Assume now that the Lipschitz condition in the statement applies to $f$. Then
		\be\nn
		\E\big[(\bar{Y}^+_t)^2\big]\leq\;\E\big[\int_{t}^{+\infty}u(r)\Gamma_{t,r}(\bar{Y}^+_r)^2\,dr\big].
		\ee
		From Gronwall's lemma, $(Y_t-Y^{\prime}_t)^+=0$, $\forall t\geq0.$
	\end{proof}
\end{appendices}
%%%%%%%%%%%%%%%%%%%%%%%%%%%%%%%%%%%%%%%%%%%%%%%%%%%%%%%%%%%%%%%%%%%%%%%%%%%%%%%%%
%%%%%%%%%%%%%%%%%%%%%%%%%%%%%%%%%%%%%%%%%%%%%%%%%%%%%%%%%%%%%%%%%%%%%%%%%%%%%%%%%

\end{document}